\documentclass[a4paper,12pt]{amsart}

\newif\ifdraft
\draftfalse

\usepackage[T1]{fontenc}
\usepackage{mathtools}
\usepackage{amssymb}
\usepackage{microtype}

\usepackage[
  a4paper,
  textwidth=16cm,
  textheight=23cm,
  centering
]{geometry}

\usepackage{graphicx}
\usepackage{enumitem}
\usepackage{xcolor}

\allowdisplaybreaks[2]

\ifdraft
  \usepackage{showkeys}
\fi

\definecolor{linkblue}{HTML}{1F4E79}
\usepackage[
  unicode,
  pdfusetitle,
  colorlinks=true,
  linkcolor=linkblue,
  citecolor=linkblue,
  urlcolor=linkblue
]{hyperref}
\usepackage[nameinlink,noabbrev]{cleveref}
\numberwithin{equation}{section}
\numberwithin{figure}{section}
\numberwithin{table}{section}

\theoremstyle{plain}
\newtheorem{theorem}{Theorem}[section]
\newtheorem{lemma}[theorem]{Lemma}
\newtheorem{proposition}[theorem]{Proposition}
\newtheorem{corollary}[theorem]{Corollary}
\newtheorem*{hall}{Hall's Marriage Theorem}

\theoremstyle{definition}
\newtheorem{definition}[theorem]{Definition}

\theoremstyle{remark}

\newcommand{\bbZ}{\mathbb{Z}}
\newcommand{\bbR}{\mathbb{R}}

\DeclareMathOperator{\Prob}{Prob}
\DeclareMathOperator{\Cov}{Cov}
\DeclareMathOperator{\ord}{ord}
\DeclareMathOperator{\supp}{supp}
\DeclareMathOperator{\diam}{diam}
\DeclareMathOperator{\mesh}{mesh}

\newcommand{\cM}{\mathcal M}
\newcommand{\cR}{\mathcal R}
\newcommand{\cE}{\mathcal E}
\newcommand{\cU}{\mathcal U}
\newcommand{\cV}{\mathcal V}
\newcommand{\cW}{\mathcal W}
\newcommand{\cA}{\mathcal A}
\newcommand{\Deltaq}{\Delta^{q-1}}

\title[Sharp Embedding Theorem]{A Sharp Embedding Theorem for Topological
Dynamical Systems}

\author{Ruxi Shi}
\address{Shanghai Center for Mathematical Sciences, Fudan University,
Shanghai 200438, China}
\email{ruxishi@fudan.edu.cn}
\thanks{This work was supported by the NSFC Nos.~12571198 and 12231013.}


\begin{document}

\begin{abstract}
Let $(X,T)$ be a topological dynamical system, and let $\dim(X,T)$ denote
Meyerovitch's dynamical dimension.  We prove that, for every integer $n\geq1$,
if $\dim(X,T)<n/2$, then the set of maps $f\in C(X,[0,1]^n)$ for which
the orbit map
\[
   I_f\colon X\longrightarrow([0,1]^n)^{\bbZ},
   \qquad I_f(x)=\bigl(f(T^k x)\bigr)_{k\in\bbZ},
\]
is a topological embedding is a dense $G_\delta$ set.
Consequently, $(X,T)$ embeds into
$(([0,1]^n)^{\bbZ},\sigma)$ for some integer $n\geq1$ if and only if
$\dim(X,T)<\infty$.  The constant $1/2$ and the strict
inequality are optimal.
\end{abstract}

\maketitle

\section{Introduction}\label{sec:introduction}

The classical embedding problem provides the model for the dynamical
question.  For a compact metrizable space $K$,
Lebesgue covering dimension gives a complete answer at the level of
finite-dimensional embeddability:
\[
   K\text{ embeds into }[0,1]^m\text{ for some }m<\infty
   \quad\Longleftrightarrow\quad
   \dim K<\infty.
\]
Indeed, if $\dim K=d<\infty$, the Menger--N\"obeling theorem embeds
$K$ into $[0,1]^{2d+1}$ \cite[Theorem~V.2]{HurewiczWallman}, whereas
every compact subspace of a finite-dimensional cube has finite covering
dimension.  The factor two cannot be improved
in general: for every $d\geq1$ there are $d$-dimensional compacta that
do not embed into $\bbR^{2d}$.  Thus covering dimension both
detects finite-dimensional embeddability and gives the sharp universal
relation between source and target dimensions.

The dynamical problem asks for an analogue of this picture.  Let
$(X,T)$ be a topological dynamical system.  The natural universal models
with finitely many coordinates per time are the cubical shifts
\[
   \bigl(([0,1]^n)^{\bbZ},\sigma\bigr),
   \qquad \text{where}\qquad
   (\sigma z)_k=z_{k+1}.
\]
Although the alphabet has dimension $n$, the shift space itself has
infinite covering dimension; the relevant resource is the number of
coordinates available at each unit of time.  Every continuous
equivariant map to the shift is the orbit map
\[
   I_f(x)=\bigl(f(T^kx)\bigr)_{k\in\bbZ}
\]
of a single map $f\colon X\to[0,1]^n$ (see
Section~\ref{sec:basic-notation}).  We therefore seek a dynamical
counterpart of $\dim K$: an invariant whose finiteness detects
embeddability into a cubical shift with finite-dimensional alphabet and
whose value controls the sharp universal bound on alphabet dimension.

Mean dimension was the first natural asymptotic candidate.  Introduced by
Gromov \cite{Gromov1999} and developed systematically by Lindenstrauss
and Weiss \cite{LindenstraussWeiss2000}, it measures the asymptotic
amount of topological dimension per iterate.  It has the basic
properties required of an embedding obstruction: it is monotone under
equivariant embeddings and
\[
   \operatorname{mdim}\bigl(([0,1]^n)^{\bbZ},\sigma\bigr)=n.
\]
Hence an embedding into the $n$-cubical shift forces
$\operatorname{mdim}(X,T)\leq n$ \cite{LindenstraussWeiss2000}.  In the
converse direction,
Lindenstrauss proved an embedding theorem under
$\operatorname{mdim}(X,T)<n/36$ for extensions of infinite minimal
topological dynamical systems
\cite[Theorem~5.1]{Lindenstrauss1999}.

Mean dimension does not record the periodic-point obstruction.  If $(X,T)$
embeds into the $n$-cubical shift, then for each $k\geq1$ the restriction
to
\[
   \operatorname{Fix}(T^k):=\{x\in X:T^kx=x\}
\]
embeds into
$\operatorname{Fix}(\sigma^k)\cong([0,1]^n)^k$.  Thus every
periodic-point set carries its own classical finite-dimensional
embedding problem.  Here and below, $\dim\varnothing=-1$.  Guided by the
Menger--N\"obeling theorem,
Lindenstrauss and Tsukamoto conjectured that the two general-position
bounds
\[
   \operatorname{mdim}(X,T)<\frac n2,
   \qquad
   \frac{\dim\operatorname{Fix}(T^k)}{k}<\frac n2
   \quad(k\geq1)
\]
should suffice for an embedding into the $n$-cubical shift
\cite[Conjecture~1.2]{LindenstraussTsukamoto2014}.  A sequence of
results strongly supported the one-half threshold in broad aperiodic
settings: it was proved for extensions of aperiodic finite-alphabet
subshifts \cite{GutmanTsukamoto2014}, for minimal systems and extensions
of nontrivial minimal systems
\cite[Theorem~1.4 and Corollary~3.5]{GutmanTsukamoto2020}, and, more
generally, for systems with the marker property
\cite{GutmanQiaoTsukamoto2019}.

The general conjecture nevertheless failed.  Recently, Dranishnikov and Levin
constructed a free $\bbZ$-action by isometries on a compact metric space
which has mean dimension zero but does not embed into any cubical shift
with finite-dimensional alphabet \cite{DranishnikovLevin2025}.  The
example satisfies both conjectured bounds for every alphabet dimension
and still admits no such embedding.  Thus mean dimension and
periodic-point data do not characterize shift embeddability.

Very recently, Meyerovitch introduced a new notion of dimension, denoted by
$\dim(X,T)$, that captures this missing information \cite{Meyerovitch2026}.
This invariant has the basic properties suggested by the classical and
dynamical embedding problems.  It
is monotone under equivariant embeddings
\cite[Proposition~3.1]{Meyerovitch2026}; for an identity action it
reduces directly from the definition to Lebesgue covering dimension;
for amenable group actions it satisfies
\[
   \operatorname{mdim}(X,T)\leq\dim(X,T)
\]
\cite[Theorem~9.1]{Meyerovitch2026}; and for every $k\geq1$ it satisfies
\[
   \frac{\dim\operatorname{Fix}(T^k)}{k}\leq\dim(X,T),
\]
so it incorporates the periodic-point constraints
\cite[Introduction, item~(5)]{Meyerovitch2026}.  It also assigns
dimension $n$ to the $n$-cubical shift
\cite[Theorem~8.1]{Meyerovitch2026}.  Moreover, the
Dranishnikov--Levin example has dynamical dimension $+\infty$
\cite[Corollary~10.2]{Meyerovitch2026}.  For systems with the marker
property (as well as uniform Rokhlin property), Meyerovitch showed
\[
   \dim(X,T)=\operatorname{mdim}(X,T)
\]
\cite[discussion preceding Lemma~9.1 and Theorem~9.2]{Meyerovitch2026}.
These facts identify dynamical dimension as the relevant obstruction in
the present problem.

Meyerovitch further proved that the bound $\dim(X,T)<n/2$ gives a dense
$G_\delta$ set of almost-embedding maps for actions of arbitrary
countable groups \cite[Theorem~10.1]{Meyerovitch2026}.  Such a map is a
measure-theoretic isomorphism for every invariant Borel probability
measure, but it need not separate every pair of points topologically.
For distal systems, almost embeddings are genuine topological
embeddings \cite[Proposition~10.2]{Meyerovitch2026}.  The remaining
question for general systems was therefore topological: does the same
bound $\dim(X,T)<n/2$ force a genuine embedding into the $n$-cubical
shift?

This paper answers that question affirmatively for arbitrary
$\bbZ$-actions, without assuming freeness, the marker property, or
distality.  The formal definition of $\dim(X,T)$ is recalled in
Section~\ref{sec:dynamical-dimension}, and $C(X,[0,1]^n)$ is equipped with the
uniform topology.  Our main result is as follows.

\begin{theorem}\label{thm:main}
Let $(X,T)$ be a topological dynamical system, and let
$n\geq1$ be an integer.  If
\[
   \dim(X,T)<\frac n2,
\]
then the set
\[
   \bigl\{f\in C(X,[0,1]^n):I_f\text{ is a topological embedding}\bigr\}
\]
is a dense $G_\delta$ subset of $C(X,[0,1]^n)$.  In particular,
\[
   (X,T)\hookrightarrow(([0,1]^n)^{\bbZ},\sigma).
\]
\end{theorem}

\begin{corollary}\label{cor:finite-dim-shift}
A topological dynamical system embeds into
$(([0,1]^n)^{\bbZ},\sigma)$ for some integer $n\geq1$ if and only if its
dynamical dimension is finite.
\end{corollary}

Although Corollary~\ref{cor:finite-dim-shift} follows from
Theorem~\ref{thm:main}, we first give a simpler proof that does not track
the optimal alphabet dimension.  This relatively simple proof illustrates
the main idea of the argument.

\begin{theorem}\label{thm:warmup}
Let $(X,T)$ be a topological dynamical system.  If
\[
   \dim(X,T)<\infty,
\]
then there is an integer $n\geq1$ for which
\[
   \bigl\{f\in C(X,[0,1]^n):I_f\text{ is a topological embedding}\bigr\}
\]
is a dense $G_\delta$ subset of $C(X,[0,1]^n)$.
\end{theorem}

Theorem~\ref{thm:warmup} proves one implication in
Corollary~\ref{cor:finite-dim-shift}.  The converse follows from the
monotonicity of dynamical dimension and the formula
\[
   \dim\bigl(([0,1]^n)^{\bbZ},\sigma\bigr)=n
\]
\cite[Proposition~3.1 and Theorem~8.1]{Meyerovitch2026}; see the end of
Section~\ref{sec:warmup}.

Dynamical dimension thus plays for equivariant shift embeddings the role
of Lebesgue covering dimension for embeddings of compacta.  For a
prescribed alphabet dimension, Theorem~\ref{thm:main} gives the optimal
universal bound
\[
   \dim(X,T)<\frac n2
   \quad\Longrightarrow\quad
   (X,T)\hookrightarrow(([0,1]^n)^{\bbZ},\sigma).
\]
Moreover, the maps whose orbit maps are embeddings form a dense $G_\delta$ set.

The threshold is sharp.  For the identity
action on a compact metrizable space $K$, one has
$\dim(K,\mathrm{id})=\dim K$, and an equivariant orbit map into the
$n$-cubical shift is an embedding precisely when its time-zero
coordinate embeds $K$ into $[0,1]^n$.  Thus, when $\dim K=d<\infty$,
Theorem~\ref{thm:main} with $n=2d+1$ recovers the classical
Menger--N\"obeling bound.  In the opposite direction,
Proposition~\ref{prop:sharp} gives, for every $D\geq1$, an identity-action
system of dynamical dimension $D$ which does not embed into the
$2D$-cubical shift.  Proposition~\ref{prop:minimal-sharp} shows that the
same boundary failure occurs in genuine dynamics: for every integer $n\geq1$
there is a minimal system with $\dim(X,T)=n/2$ which does not embed into
the $n$-cubical shift.  Hence the constant $1/2$ cannot be increased,
and the strict inequality cannot be replaced by a non-strict one, even
among minimal systems.

For systems with the marker property, dynamical dimension equals mean
dimension, so Theorem~\ref{thm:main} recovers the sharp condition
$\operatorname{mdim}(X,T)<n/2$.  The theorem itself assumes neither the
marker property, freeness, nor distality, and therefore also applies to
marker-free systems of finite dynamical dimension; an explicit example
is studied in \cite{Shi2026}.  It also strengthens Meyerovitch's almost
embedding conclusion to pointwise injectivity on all of $X$.

\subsection{Idea of the proof}

We describe the proof of Theorem~\ref{thm:warmup}; the proof of
Theorem~\ref{thm:main} follows the same idea.

The proof begins with a coverwise separation problem.  For a finite open
cover $\cU$ of $X$, let $\mathcal G(\cU)$ be the set of maps $f$
such that
\[
   I_f(x)=I_f(y)\quad\Longrightarrow\quad
   \{x,y\}\subset U\quad\text{for some }U\in\cU.
\]
Compactness shows that $\mathcal G(\cU)$ is open.  If it is dense for
every $\cU$, the Baire category theorem, applied to covers whose meshes
tend to zero, gives a residual set of injective orbit maps.  This
reduction is proved in Section~\ref{sec:coverwise-baire}.

The density argument uses a simplex as an intermediate space.  Given $f_0$, $\cU$,
and $\rho>\dim(X,T)$, choose a partition of unity subordinate to a
suitable refinement and let $\phi$ be the resulting simplex-valued map.
We obtain an upper semicontinuous function $\kappa$, for which
Lemma~\ref{lem:orbit-sum} gives the uniform bound
\[
   \sum_{j\in I}\kappa(T^j x)\leq\rho|I|+C
\]
for every $x\in X$ and every finite integer interval $I$.  A local affine
function of the code approximates $f_0$; it remains close to $f_0$ after
a sufficiently small finite-band convolution perturbation. So new we map $X$ equivarently and $\mathcal{U}$-injectively to a subset of $(\mathbb{R}^q)^\mathbb{Z}$ which has the density of non-zero elements bounded by $\rho$. But the number $q$ would be much larger than $n$. Then we need to find a generic linear map from $(\mathbb{R}^q)^\mathbb{Z}$ to $(\mathbb{R}^n)^\mathbb{Z}$ which is injective on the subset mentioned before. In order to do this, we localize the linear map from infinite dimensional space to finite one.

Choose $n$ with $2\rho+2<n$ and retain every
nonzero standard coordinate of
\[
   w_j:=\phi(T^j x)-\phi(T^j y).
\]
Put $E_j:=\{i:w_{j,i}\neq0\}$ and $k_j:=|E_j|$.  Then
\[
   k_j\leq\kappa(T^j x)+\kappa(T^j y)+2,
   \qquad
   \sum_{j\in I}k_j\leq(2\rho+2)|I|+2C.
\]
Lemma~\ref{lem:finite-window-matching} gives integer intervals
$J=[A,B]$ and $K=[A+R,B-R]$, with $0\in J$, and an $R$-injective map
\[
   \pi\colon E:=\{(j,i):j\in J,\ i\in E_j\}
   \longrightarrow K\times\{1,\ldots,n\}.
\]
Let $\cR:=\pi(E)$.  For a coefficient family
$M=(M_{-R},\ldots,M_R)$, set
\[
   (\cM_Mz)_k:=\sum_{r=-R}^{R}M_rz_{k+r}
\]
and form the square matrix
\[
   A_J(M):=
   \bigl[(\cM_M\mathbf e_{j,i})_{k,c}\bigr]_{
      (k,c)\in\cR,\ (j,i)\in E},
\]
where $\mathbf e_{j,i}$ is supported at time $j$ and has value $e_i$
there.  We choose $M$ so that every matrix of this form is invertible.
Such a coefficient family exists by
Proposition~\ref{prop:generic-regular}.
If $I_f(x)=I_f(y)$, then $\cM_Mw=0$.  Since $K$ is the $R$-interior of
$J$, restriction to the coordinates in $\cR$ gives
\[
   A_J(M)(w_{j,i})_{(j,i)\in E}=0.
\]
The vector on the left is nonzero, whereas $A_J(M)$ is invertible, a
contradiction.

\subsection{Organization of the paper}

Section~\ref{sec:preliminaries} recalls the basic definitions.
Section~\ref{sec:orbit-sum-and-coding} proves the Baire reduction,
establishes the uniform orbit-sum estimate, and constructs the simplex code.
Section~\ref{sec:warmup} proves the qualitative
theorem, including the finite-window matching lemma.
Section~\ref{sec:sharp-upgrade} develops the signed-difference
decomposition and establishes the generic superregularity needed
for the sharp count.  Section~\ref{sec:sharp-proof} proves
Theorem~\ref{thm:main}.  Sections~\ref{sec:applications} and
\ref{sec:sharpness} give the applications and establish sharpness,
respectively.

\section{Preliminaries}
\label{sec:preliminaries}

We begin by fixing notation and recalling Meyerovitch's dynamical dimension.

\subsection{Basic notation}\label{sec:basic-notation}

Throughout the paper, $(X,T)$ is a topological dynamical system: $X$ is a
nonempty compact metrizable space, $d_X$ is a compatible metric on $X$,
and the map $T\colon X\to X$ is a homeomorphism.  For each positive
integer $n$, on $\bbR^n$ we use the
maximum norm, and on the corresponding function space we use the
uniform norm:
\[
   \|v\|_\infty=\max_{1\leq c\leq n}|v_c|,
   \qquad
   \|f-g\|_\infty=\sup_{x\in X}\|f(x)-g(x)\|_\infty.
\]
Diameters of subsets of $\bbR^n$ are taken in the maximum norm.
For positive integers $n,q$, let $M_{n\times q}(\bbR)$ denote the space
of real $n\times q$ matrices.  For $A\in M_{n\times q}(\bbR)$, put
\[
   \|A\|_{\max}=\max_{c,i}|A_{c,i}|.
\]
Also put $\mathbf 1_n=(1,\ldots,1)\in\bbR^n$.

For a topological space $Y$, let
\[
   \Cov(Y):=\{\cU:\cU\text{ is a finite open cover of }Y\}.
\]
For $\cU\in\Cov(X)$, set
\[
   \mesh(\cU)=\max_{U\in\cU}\sup_{x,y\in U}d_X(x,y).
\]
For $\cU\in\Cov(X)$ and a map $F\colon X\to Y$, we say that $F$ is
\emph{$\cU$-separating} if, for every $x,y\in X$,
\[
   F(x)=F(y)\quad\Longrightarrow\quad
   \{x,y\}\subset U\quad\text{for some }U\in\cU.
\]

For any space $Y$, the (left-)shift on $Y^{\bbZ}$ is
\[
   (\sigma z)_k=z_{k+1}
   \qquad(z\in Y^{\bbZ}).
\]
For a positive integer $n$ and $f\in C(X,[0,1]^n)$, define
\[
   I_f\colon X\longrightarrow([0,1]^n)^{\bbZ},
   \qquad
   I_f(x)_k=f(T^k x).
\]
Then $I_f\circ T=\sigma\circ I_f$.  Every continuous equivariant map from $(X,T)$
to this shift is the orbit map of its zeroth-coordinate function.

For a function $\varphi\colon X\to\bbR$, write
\[
   \supp\varphi=\overline{\{x\in X:\varphi(x)\ne0\}}.
\]
For every integer $q\geq1$, let
\[
   \Deltaq
   :=\left\{p=(p_1,\ldots,p_q)\in[0,1]^q:
             \sum_{i=1}^q p_i=1\right\}
\]
be the standard $(q-1)$-simplex.

\subsection{Meyerovitch's dynamical dimension}
\label{sec:dynamical-dimension}

Let $\Prob(X)$ be the space of Borel probability measures on $X$ with
the weak-$*$ topology, and let
\[
   \Prob(X,T)=\{\mu\in\Prob(X):T_*\mu=\mu\},
   \qquad
   T_*\mu(E)=\mu(T^{-1}E)
\]
for every Borel set $E$.  The set $\Prob(X,T)$ is nonempty, compact, and
convex.  For $\cV\in\Cov(X)$, put
\[
\begin{aligned}
   \ord(\cV,x)&=\#\{V\in\cV:x\in V\}-1,\\
   \ord(\cV,T)
   &=\sup_{\mu\in\Prob(X,T)}
      \int_X\ord(\cV,x)\,d\mu(x).
\end{aligned}
\]
Since $\ord(\cV,\cdot)$ is a finite sum of indicators of open sets, it is
lower semicontinuous, so the integral is well defined.
For $\cU,\cV\in\Cov(X)$, write $\cV\succeq\cU$ when $\cV$ refines
$\cU$, meaning that every $V\in\cV$ is contained in some $U\in\cU$.
Define
\[
   \dim(\cU,T)
   =\inf_{\substack{\cV\in\Cov(X)\\
                    \cV\succeq\cU}}
      \ord(\cV,T),
   \qquad
   \dim(X,T)
   =\sup_{\cU\in\Cov(X)}\dim(\cU,T).
\]
This is Meyerovitch's dynamical dimension
\cite[equation~(2)]{Meyerovitch2026}.

\section{Technical tools}
\label{sec:orbit-sum-and-coding}

\subsection{Baire reduction}
\label{sec:coverwise-baire}

Fix a positive integer $n$.  For $\cU\in\Cov(X)$, set
\[
   \mathcal G(\cU)
   :=\bigl\{f\in C(X,[0,1]^n):I_f\text{ is }\cU\text{-separating}\bigr\}.
\]
The next lemma and proposition reduce both embedding theorems to proving
that $\mathcal G(\cU)$ is dense for every finite open cover $\cU$.

\begin{lemma}\label{lem:GU-open}
For every $\cU\in\Cov(X)$, the set $\mathcal G(\cU)$ is open in
$C(X,[0,1]^n)$.
\end{lemma}

\begin{proof}
Put
\[
   P_{\cU}:=(X\times X)\setminus\bigcup_{U\in\cU}(U\times U).
\]
This is a compact subset of $X\times X$.  If $P_{\cU}=\varnothing$, then
$\mathcal G(\cU)=C(X,[0,1]^n)$.  Otherwise, let
$f\in\mathcal G(\cU)$.  For $N\geq0$, define
\[
   D_N^f(x,y):=\max_{|k|\leq N}
      \|f(T^k x)-f(T^k y)\|_\infty.
\]
For every $(x,y)\in P_{\cU}$, the orbit maps $I_f(x)$ and $I_f(y)$ are
distinct, so $D_N^f(x,y)>0$ for some $N$.  The open sets
$\{D_N^f>0\}$ increase with $N$ and cover $P_{\cU}$.  Compactness gives
one $N$ for which $D_N^f>0$ throughout $P_{\cU}$, and hence
\[
   \delta:=\min_{(x,y)\in P_{\cU}}D_N^f(x,y)>0.
\]
If $g\in C(X,[0,1]^n)$, then
\[
   D_N^g(x,y)\geq D_N^f(x,y)-2\|g-f\|_\infty.
\]
Thus $\|g-f\|_\infty<\delta/3$ implies $g\in\mathcal G(\cU)$.
\end{proof}

\begin{proposition}\label{prop:baire-reduction}
Suppose that $\mathcal G(\cU)$ is dense in $C(X,[0,1]^n)$ for every
$\cU\in\Cov(X)$.  Then
\[
   \bigl\{f\in C(X,[0,1]^n):I_f\text{ is a topological embedding}\bigr\}
\]
is a dense $G_\delta$ subset of $C(X,[0,1]^n)$.
\end{proposition}

\begin{proof}
Choose finite open covers $(\cU_m)_{m\geq1}$ with
$\mesh(\cU_m)\to0$; for example, take a finite subcover of the cover by
open $d_X$-balls of radius $1/(2m)$, so that
$\mesh(\cU_m)\leq1/m$.  By
Lemma~\ref{lem:GU-open} and the assumed
density, every $\mathcal G(\cU_m)$ is open and dense in the complete
metric space $C(X,[0,1]^n)$.  It follows that
\[
   \bigcap_{m=1}^{\infty}\mathcal G(\cU_m)
\]
is a dense $G_\delta$ set.  If $f$ belongs to this intersection and
$I_f(x)=I_f(y)$, then $x$ and $y$ lie in a common member of every
$\cU_m$, so
\[
   d_X(x,y)\leq\mesh(\cU_m)\longrightarrow0.
\]
Thus $x=y$.  We conclude that
\[
   \bigl\{f\in C(X,[0,1]^n):I_f\text{ is a topological embedding}\bigr\}
\]
is a dense $G_\delta$ subset of $C(X,[0,1]^n)$.
\end{proof}

\subsection{A uniform orbit-sum estimate}\label{sec:orbit-sum}

The following lemma converts an invariant-measure estimate into a uniform
estimate on all finite orbit segments.

\begin{lemma}\label{lem:orbit-sum}
Let $u\colon X\to[0,\infty)$ be bounded and upper semicontinuous.  For every
\[
   \rho>\sup_{\mu\in\Prob(X,T)}\int_Xu\,d\mu
\]
there is a constant $C\geq0$ such that
\begin{equation}
   \sum_{j=s}^{t}u(T^j x)
   \leq \rho(t-s+1)+C\label{eq:2.1}
\end{equation}
for every $x\in X$ and every pair of integers $s\leq t$.
\end{lemma}

\begin{proof}
We first claim that there is a positive integer $N_0$ such that
\begin{equation}
   \frac1N\sum_{j=0}^{N-1}u(T^j x)<\rho\label{eq:2.2}
\end{equation}
for every $x\in X$ and every $N\geq N_0$.

If this were false, there would be a sequence
$(N_\ell)_{\ell\geq1}$ of positive integers and a sequence
$(x_\ell)_{\ell\geq1}$ in $X$ such that $N_\ell\to\infty$ and
\[
   \frac1{N_\ell}\sum_{j=0}^{N_\ell-1}u(T^j x_\ell)\geq\rho
   \qquad(\ell\geq1).
\]
Let
\[
   \nu_\ell=\frac1{N_\ell}
      \sum_{j=0}^{N_\ell-1}\delta_{T^j x_\ell}
   \qquad(\ell\geq1).
\]
Here $\delta_y$ denotes the Dirac probability measure at $y\in X$.
After passing to a subsequence, $\nu_\ell$ converges weak-$*$ to a
probability measure $\nu$.  For every $g\in C(X)$,
\[
\begin{aligned}
 \left|\int_Xg\circ T\,d\nu_\ell-\int_Xg\,d\nu_\ell\right|
 &=\frac1{N_\ell}
   \left|g(T^{N_\ell}x_\ell)-g(x_\ell)\right|\\
 &\leq\frac{2\|g\|_\infty}{N_\ell}\longrightarrow0
 \qquad\text{as }\ell\to\infty.
\end{aligned}
\]
Passing to the limit gives $T_*\nu=\nu$.  Since $u$ is bounded and upper
semicontinuous, the map
$\mu\mapsto\int_Xu\,d\mu$ is upper semicontinuous for the weak-$*$
topology by the Portmanteau theorem
\cite[Theorem~2.1]{Billingsley1999}.  Hence
\[
   \rho
   \leq\limsup_{\ell\to\infty}\int_Xu\,d\nu_\ell
   \leq\int_Xu\,d\nu
   \leq\sup_{\mu\in\Prob(X,T)}\int_Xu\,d\mu,
\]
a contradiction.  Thus \eqref{eq:2.2} holds.

Given $s\leq t$, if $t-s+1\geq N_0$, applying \eqref{eq:2.2} to
$T^s x$ gives
\[
   \sum_{j=s}^{t}u(T^j x)<\rho(t-s+1).
\]
If $t-s+1<N_0$, then, since $u\geq0$ and $\rho>0$,
\[
   \sum_{j=s}^{t}u(T^j x)-\rho(t-s+1)
   \leq (t-s+1)\sup_{x\in X}u(x)
   \leq (N_0-1)\sup_{x\in X}u(x),
\]
so \eqref{eq:2.1} holds in all cases with
$C:=(N_0-1)\sup_{x\in X}u(x)$.
\end{proof}

\subsection{Simplex coding}
\label{sec:coding-package}

The next lemma gives the simplex code used in both density arguments.
The definition of $\Lambda(x)$ uses closed supports, ensuring that
$\kappa(x)=|\Lambda(x)|-1$ is upper semicontinuous.

\begin{lemma}
\label{lem:coding-package}
Let $n$ be a positive integer, $\cU\in\Cov(X)$,
$f_0\in C(X,[0,1]^n)$, $\varepsilon>0$, and
$\rho>\dim(X,T)$.  Then there exist
\begin{itemize}[leftmargin=2em]
\item a number $0<\lambda<\min\{1/4,\varepsilon/4\}$ and
      $f_0^\circ=(1-2\lambda)f_0+\lambda\mathbf 1_n$;
\item an integer $q\geq1$, a continuous map
      $\phi\colon X\to\Deltaq$, and the equivariant code
      $I_\phi\colon X\to(\Deltaq)^{\bbZ}$ given by
      $I_\phi(x)_j=\phi(T^j x)$;
\item nonempty sets $\Lambda(x)\subseteq\{1,\ldots,q\}$ and a bounded
      upper semicontinuous function
      $\kappa(x)=|\Lambda(x)|-1$;
\item a constant $C\geq0$ and a matrix
      $B_0\in M_{n\times q}(\bbR)$;
\end{itemize}
such that the following hold.
\begin{enumerate}[label=\textup{(\roman*)},leftmargin=2.6em]
\item If $\phi_i(x)>0$, then $i\in\Lambda(x)$.
\item For every $x\in X$ and every finite integer interval $I$,
\begin{equation}\label{eq:coding-orbit-bound}
   \sum_{j\in I}\kappa(T^j x)\leq\rho|I|+C.
\end{equation}
\item With
\[
   h_0(z):=B_0z_0
   \qquad(z\in(\Deltaq)^{\bbZ}),
\]
one has
\begin{equation}\label{eq:coding-range-approximation}
\begin{aligned}
   h_0((\Deltaq)^{\bbZ})&\subset[\lambda,1-\lambda]^n,\\
   \|h_0(I_\phi(x))-f_0^\circ(x)\|_\infty&<\frac{\varepsilon}{4}
      \qquad(x\in X),\\
   \|f_0^\circ-f_0\|_\infty&\leq\lambda.
\end{aligned}
\end{equation}
\item If $x,y$ do not lie in a common member of $\cU$, then
\begin{equation}\label{eq:lambda-separation}
   \Lambda(x)\cap\Lambda(y)=\varnothing.
\end{equation}
\end{enumerate}
\end{lemma}

\begin{proof}
Choose $\lambda$ as in the statement and put
\[
   f_0^\circ=(1-2\lambda)f_0+\lambda\mathbf 1_n.
\]
Then
\[
   f_0^\circ(X)\subset[\lambda,1-\lambda]^n,
   \qquad
   \|f_0^\circ-f_0\|_\infty\leq\lambda.
\]
Choose a finite open cover $\cA$ of $X$ such that
\begin{equation}\label{eq:small-oscillation-cover}
   \diam f_0^\circ(A)<\frac{\varepsilon}{4}
   \qquad(A\in\cA).
\end{equation}
Let
\[
   \cW:=\cU\vee\cA
   =\{U\cap A:U\in\cU,\ A\in\cA\}.
\]
Since $\dim(\cW,T)\leq\dim(X,T)<\rho$, there is a finite open
refinement
\[
   \cV=\{V_1,\ldots,V_q\}\succeq\cW
\]
with no empty members and
\begin{equation}\label{eq:coding-order-mean}
   \sup_{\mu\in\Prob(X,T)}
      \int_X\ord(\cV,x)\,d\mu(x)<\rho.
\end{equation}
By the partition of unity, choose a continuous partition of unity
$\phi_1,\ldots,\phi_q$ subordinate to $\cV$:
\[
   \sum_{i=1}^{q}\phi_i=1,
   \qquad
   \supp\phi_i\subset V_i.
\]
For $x\in X$, define
\[
   \phi(x)=(\phi_1(x),\ldots,\phi_q(x)),
   \qquad
   I_\phi(x)=\bigl(\phi(T^j x)\bigr)_{j\in\bbZ},
\]
and set
\[
   \Lambda(x):=\{i:x\in\supp\phi_i\},
   \qquad
   \kappa(x):=|\Lambda(x)|-1.
\]
Observe that
\[
   \{i:\phi_i(x)>0\}\subseteq\Lambda(x).
\]
Each $\supp\phi_i$ is closed, and hence
\[
   \kappa=\sum_{i=1}^{q}\mathbf 1_{\supp\phi_i}-1
\]
is bounded and upper semicontinuous.  If $i\in\Lambda(x)$, then
$x\in\supp\phi_i\subset V_i$.  Therefore
\[
   \kappa(x)\leq\ord(\cV,x).
\]
It follows from~\eqref{eq:coding-order-mean} that
\[
   \sup_{\mu\in\Prob(X,T)}\int_X\kappa\,d\mu<\rho.
\]
Lemma~\ref{lem:orbit-sum} now gives a constant $C\geq0$ such that
\[
   \sum_{j\in I}\kappa(T^j x)\leq\rho|I|+C
\]
for every $x\in X$ and every finite integer interval $I$.

For each $1\leq i\leq q$, choose $U_i\in\cU$, $A_i\in\cA$, and $x_i\in V_i$ such
that
\[
   V_i\subset U_i\cap A_i.
\]
Let $B_0$ be the $n\times q$ matrix whose $i$-th column is
$f_0^\circ(x_i)$, i.e.,
\[
   B_0=\bigl(f_0^\circ(x_1)\ \cdots\ f_0^\circ(x_q)\bigr).
\]
Since $B_0z_0$ is a convex combination of
$f_0^\circ(x_1),\ldots,f_0^\circ(x_q)$ for every $z_0\in\Deltaq$, we have
\[
   h_0((\Deltaq)^{\bbZ})\subset[\lambda,1-\lambda]^n.
\]
If $\phi_i(x)>0$, then $x\in\supp\phi_i\subset V_i$, so both $x$ and
$x_i$ lie in $A_i$.  By~\eqref{eq:small-oscillation-cover},
\[
\begin{aligned}
   \|h_0(I_\phi(x))-f_0^\circ(x)\|_\infty
   &\leq\sum_{i:\phi_i(x)>0}
      \phi_i(x)\|f_0^\circ(x_i)-f_0^\circ(x)\|_\infty\\
   &<\frac{\varepsilon}{4}.
\end{aligned}
\]
This proves~\eqref{eq:coding-range-approximation}.

Finally, suppose $x,y$ do not lie in a common member of $\cU$.  If
$i\in\Lambda(x)\cap\Lambda(y)$, then
$x,y\in\supp\phi_i\subset V_i\subset U_i$, a contradiction.  Thus
\eqref{eq:lambda-separation} holds.
\end{proof}

\section{The qualitative embedding theorem}
\label{sec:warmup}

We first prove Theorem~\ref{thm:warmup} in this section.  The
same finite-window argument will be used for the sharp theorem, with a
more careful choice of columns.

\subsection{Regular matrix}
\label{sec:regularity}

Fix positive integers $q,n$ and a nonnegative integer $R$.  Let
\[
   M=(M_{-R},\ldots,M_R),
   \qquad
   M_r\in M_{n\times q}(\bbR).
\]
Each $M_r$ defines a linear map from $\bbR^q$ to $\bbR^n$.  Later, $q$
will be the number of sets in a refining cover, while $n$ will remain the
dimension of the shift alphabet.
For $z=(z_k)_{k\in\bbZ}\in(\bbR^q)^{\bbZ}$, define the linear map
\begin{equation}\label{eq:convolution}
\begin{aligned}
   \cM_M\colon(\bbR^q)^{\bbZ}&\longrightarrow(\bbR^n)^{\bbZ},\\
   \text{with}\qquad (\cM_Mz)_k&:=\sum_{r=-R}^{R}M_rz_{k+r}.
\end{aligned}
\end{equation}
With block rows indexed by $k$ and block columns indexed by $j$, the
$(k,j)$ block of $\cM_M$ is
\[
   [\cM_M]_{k,j}=
   \begin{cases}
      M_{j-k},&|j-k|\leq R,\\
      0,&|j-k|>R.
   \end{cases}
\]
This means that only the $(k,j)$-blocks with $|k-j|\leq R$ can be nonzero.
For $j\in\bbZ$ and $1\leq i\leq q$, let
$\mathbf e_{j,i}\in(\bbR^q)^{\bbZ}$ be the sequence whose value at time
$j$ is the $i$-th standard basis vector $e_i$ and whose other entries
vanish.  For $(k,c)\in\bbZ\times\{1,\ldots,n\}$, the corresponding
scalar entry is
\[
   (\cM_M\mathbf e_{j,i})_{k,c}
   =
   \begin{cases}
      (M_{j-k})_{c,i},&|j-k|\leq R,\\
      0,&|j-k|>R.
   \end{cases}
\]

\begin{definition}\label{def:r-injective}
Let $n$ be a positive integer and $R$ a nonnegative integer.  Let $F$
be a finite label set and let
\[
   \mathcal F\subset\bbZ\times F,
   \qquad
   \mathcal R\subset\bbZ\times\{1,\ldots,n\}
\]
be finite.  Write $\mathrm t(j,a)=j$.  An injective map
$\pi\colon\mathcal F\to\mathcal R$ is \emph{$R$-injective} if
\[
   |\mathrm t(x)-\mathrm t(\pi(x))|\leq R
   \qquad\text{for all }x\in\mathcal F.
\]
It is \emph{$R$-bijective} if, in addition, $\pi$ is bijective.
\end{definition}

\begin{lemma}\label{lem:banded-uncrossing}
Let $N$ be a positive integer and $R$ a nonnegative integer.  Let
\[
   \xi_1<\cdots<\xi_N,
   \qquad
   \eta_1<\cdots<\eta_N
\]
be real numbers, and let
\[
   j_1\leq\cdots\leq j_N,
   \qquad
   k_1\leq\cdots\leq k_N
\]
be integers.  Suppose that there exists an $R$-bijective map from
\[
   \{(j_p,\xi_p):1\leq p\leq N\}
   \longrightarrow
   \{(k_p,\eta_p):1\leq p\leq N\}.
\]
Then the order-preserving map
\[
   (j_p,\xi_p)\longmapsto(k_p,\eta_p)
   \qquad(1\leq p\leq N)
\]
is $R$-bijective and, among all $R$-bijective maps
$(j_p,\xi_p)\mapsto(k_{\pi(p)},\eta_{\pi(p)})$, uniquely minimizes
\[
   \sum_{p=1}^{N}(\eta_{\pi(p)}-\xi_p)^2.
\]
\end{lemma}

\begin{proof}
Suppose an $R$-bijective map has crossed pairs
\[
   (j_{p_1},\xi_{p_1})\longmapsto(k_{s_2},\eta_{s_2}),
   \qquad
   (j_{p_2},\xi_{p_2})\longmapsto(k_{s_1},\eta_{s_1}),
   \qquad\text{with }p_1<p_2\text{ and }s_1<s_2.
\]
We claim that exchanging the images of such a crossed pair preserves
$R$-bijectivity.  Indeed,
$j_{p_1}\leq j_{p_2}$ and $k_{s_1}\leq k_{s_2}$ imply
\[
\begin{aligned}
   j_{p_1}-k_{s_1}&\leq j_{p_2}-k_{s_1}\leq R,
   &k_{s_1}-j_{p_1}&\leq k_{s_2}-j_{p_1}\leq R,\\
   j_{p_2}-k_{s_2}&\leq j_{p_2}-k_{s_1}\leq R,
   &k_{s_2}-j_{p_2}&\leq k_{s_2}-j_{p_1}\leq R.
\end{aligned}
\]
The exchange strictly decreases
$\sum_{p=1}^{N}(\eta_{\pi(p)}-\xi_p)^2$, since
\[
\begin{aligned}
 & (\eta_{s_2}-\xi_{p_1})^2+(\eta_{s_1}-\xi_{p_2})^2
   -(\eta_{s_1}-\xi_{p_1})^2-(\eta_{s_2}-\xi_{p_2})^2\\
 &\hspace{4cm}=2(\xi_{p_2}-\xi_{p_1})(\eta_{s_2}-\eta_{s_1})>0.
\end{aligned}
\]
There are only finitely many $R$-bijective maps, and each uncrossing
strictly decreases the cost.  The procedure therefore terminates at the
unique inversion-free map
$(j_p,\xi_p)\mapsto(k_p,\eta_p)$ for $1\leq p\leq N$.  A
minimizing $R$-bijective map cannot contain an inversion, proving
uniqueness.
\end{proof}

\begin{definition}
\label{def:regular}
Let $q,n$ be positive integers and $R$ a nonnegative integer.  A
coefficient family
\[
   M=(M_{-R},\ldots,M_R)
   \in\bigl(M_{n\times q}(\bbR)\bigr)^{2R+1}
\]
is
\emph{regular} if the following holds.  Whenever
$E\subset\bbZ\times\{1,\ldots,q\}$ is finite and nonempty,
$\mathcal R\subset\bbZ\times\{1,\ldots,n\}$ is finite,
$|E|=|\mathcal R|$, and there is an $R$-bijective map from $E$ to
$\mathcal R$, the square matrix
\[
   \bigl[(\cM_M\mathbf e_{j,i})_{k,c}\bigr]_{
      (k,c)\in\mathcal R,\ (j,i)\in E}
\]
is invertible.
\end{definition}

\begin{proposition}
\label{prop:generic-regular}
For positive integers $q,n$ and a nonnegative integer $R$, the
regular coefficient families form a dense $G_\delta$ subset of
$\bigl(M_{n\times q}(\bbR)\bigr)^{2R+1}$.
\end{proposition}

\begin{proof}
Fix nonempty finite sets $E$ and $\mathcal R$ of equal cardinality for which an
$R$-bijective map exists.  The determinant of the associated square
matrix is a polynomial in the entries of $M_{-R},\ldots,M_R$.  We show
that it is not the zero polynomial.

Choose an integer $L>q+n$ and specialize $M$ to $\widetilde M$ by setting
\[
   (\widetilde M_r)_{c,i}
   :=X^{(-Lr+q+c-i)^2}
   \qquad(|r|\leq R).
\]
For $(j,i)\in E$ and $(k,c)\in\mathcal R$, set
\[
   \xi_{j,i}:=Lj+i,
   \qquad
   \eta_{k,c}:=Lk+q+c.
\]
Since $L>q+n$,
\[
\begin{aligned}
   \xi_{j,i}<\xi_{j',i'}
   &\iff j<j'\ \text{or}\ (j=j'\ \text{and}\ i<i'),\\
   \eta_{k,c}<\eta_{k',c'}
   &\iff k<k'\ \text{or}\ (k=k'\ \text{and}\ c<c').
\end{aligned}
\]
Moreover,
\[
   (-L(j-k)+q+c-i)^2=(\eta_{k,c}-\xi_{j,i})^2.
\]
Thus, for every bijection $\pi\colon E\to\mathcal R$,
\[
   \prod_{(j,i)\in E}
      (\cM_{\widetilde M}\mathbf e_{j,i})_{\pi(j,i)}
   =
   \begin{cases}
      X^{\sum_{(j,i)\in E}
         (\eta_{\pi(j,i)}-\xi_{j,i})^2},
         &\pi\text{ is $R$-bijective},\\
      0,&\text{otherwise}.
   \end{cases}
\]
By Lemma~\ref{lem:banded-uncrossing}, the order-preserving bijection is
$R$-bijective and uniquely minimizes
$\sum_{(j,i)\in E}(\eta_{\pi(j,i)}-\xi_{j,i})^2$.
Hence the specialized
determinant has a unique lowest-degree monomial, with coefficient
$\pm1$.  It is nonzero, so the original determinant polynomial is
nonzero.

There are only countably many such nonempty pairs $(E,\mathcal R)$.  For each one, the
complement of the zero set of its nonzero determinant polynomial is
open and dense.  Their countable intersection is the required dense
$G_\delta$ set.
\end{proof}

\subsection{A finite-window matching lemma}
\label{sec:matching}

For integers $a\leq b$, write $[a,b]=\{a,a+1,\ldots,b\}\subset\bbZ$.
We use Hall's theorem in the following standard form.

\begin{hall}[\cite{Hall1935}]
Let $G$ be a finite bipartite graph with bipartition $(V_L,V_R)$.  For
$S\subseteq V_L$, let $\Gamma_G(S)$ be its neighborhood.  Then $G$ has
a matching covering $V_L$ if and only if
$|S|\leq|\Gamma_G(S)|$ for every $S\subseteq V_L$.
\end{hall}

The parameter $\iota\in\{0,1\}$ records whether an extra formal column
is present.  We take $\iota=0$ in the qualitative proof and
$\iota=1$ in the sharp proof.

\begin{lemma}\label{lem:finite-window-matching}
Let $n$ be a positive integer, let $(k_j)_{j\in\bbZ}$ be nonnegative
integers, and suppose that there are constants $0\leq\alpha<n$ and
$C\geq0$ such that
\begin{equation}\label{eq:window-density}
   \sum_{j\in I}k_j\leq\alpha|I|+C
\end{equation}
for every finite integer interval $I$.  Let $D\subset\bbZ$ be nonempty,
let $\iota\in\{0,1\}$, and choose an integer $R\geq1$ such that
\[
   2nR>C+\iota.
\]
Then there are integers $A\leq B$ such that, with
\[
   J=[A,B],
   \qquad
   D\cap J\ne\varnothing,
   \qquad
   d_*:=\min(D\cap J),
   \qquad
   K=[A+R,B-R],
\]
one has $|J|>2R$; in particular, $K$ is nonempty.  Put
\[
   \mathcal S_\iota=
   \begin{cases}
      \varnothing,&\iota=0,\\
      \{(d_*,\star)\},&\iota=1,
   \end{cases}
\]
and define
\[
   \mathcal L_J
   :=\{(j,r):j\in J,\ 1\leq r\leq k_j\}
      \mathbin{\sqcup}\mathcal S_\iota.
\]
Then $\mathcal L_J$ admits an $R$-injective map into
$K\times\{1,\ldots,n\}$.
\end{lemma}

Figure~\ref{fig:matching-lemma} illustrates the matching when the extra
formal column is present.
\begin{figure}[htbp]
   \centering
   \includegraphics[width=0.92\textwidth]{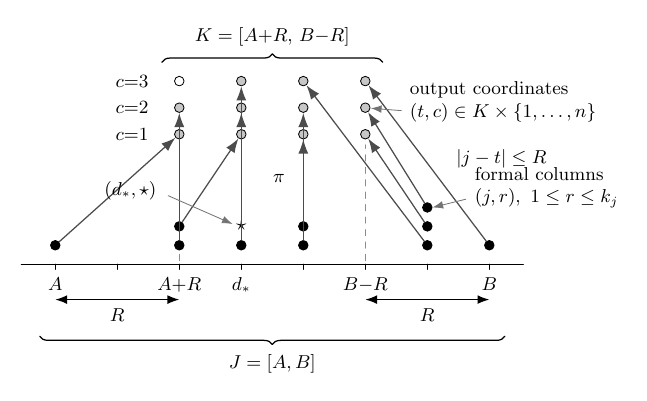}
   \caption{The finite-window matching when $n=3$, $R=2$, and
   $\iota=1$.  Filled dots represent the columns $(j,r)$, the star is
   the additional column $(d_*,\star)$, and the circles are the output
   coordinates over $K=[A+R,B-R]$.  Each arrow changes the time by at
   most $R$.  For $\iota=0$, the star and its arrow are omitted.}
   \label{fig:matching-lemma}
\end{figure}

\begin{proof}
\textbf{Step 1: Window selection.}
Define an integer-valued function $Q\colon\bbZ\to\bbZ$ by $Q(0)=0$ and
\[
   Q(t)-Q(s)=\sum_{j=s}^{t-1}(n-k_j)
   \qquad(s<t).
\]
By~\eqref{eq:window-density},
\begin{equation}\label{eq:window-growth}
   Q(t)-Q(s)\geq(n-\alpha)(t-s)-C.
\end{equation}
Taking $s=0$ or $t=0$ in~\eqref{eq:window-growth} shows that
\begin{equation}\label{eq:window-limits}
   Q(t)\longrightarrow+\infty\quad(t\to+\infty),
   \qquad
   Q(t)\longrightarrow-\infty\quad(t\to-\infty).
\end{equation}

Choose $d_0\in D$.  By~\eqref{eq:window-limits}, the function $Q$ attains a minimum
$m$ on $\{t\in\bbZ:t\geq d_0+1\}$.  Choose $N<d_0$ with $Q(N)<m$, and
let $A$ be the largest minimizer of $Q$ on
$\{t\in\bbZ:t\geq N\}$.  This minimizer exists, and there are only
finitely many of them, because $Q(t)\to+\infty$ as $t\to+\infty$.
Moreover, $A\leq d_0$: otherwise $Q(A)\geq m>Q(N)$, contrary to the
choice of $A$.  Since $A$ is the last minimizer, we have
\begin{equation}\label{eq:window-left-record}
   Q(t)>Q(A)\qquad(t>A).
\end{equation}
In particular, we obtain the well-defined integer
\[
   d_*:=\min(D\cap[A,\infty))
\]
Choose $L_0\geq\max\{A,d_*\}$ so large that
\[
   (n-\alpha)(L_0-A+1)>2nR+C+\iota.
\]
Put $M_0=\max_{A\leq t\leq L_0}Q(t)$, let $u>L_0$ be the first integer
with $Q(u)>M_0$, and set $B=u-1$.  Such a $u$ exists by
\eqref{eq:window-limits}.  The definition of $u$ gives
\begin{equation}\label{eq:window-right-record}
   Q(B+1)>Q(t)\qquad(A\leq t\leq B),
\end{equation}
and $B\geq L_0\geq d_*$.  Since $B\geq L_0$, the choice of $L_0$
also gives
\begin{equation}\label{eq:window-global-slack}
   (n-\alpha)|J|>2nR+C+\iota,
   \qquad
   |J|>2R.
\end{equation}
Here $|J|>2R$ follows from $0<n-\alpha\leq n$.  Finally,
$B\geq d_*$ and the definition of $d_*$ show that
$d_*=\min(D\cap J)$.

\textbf{Step 2: Boundary estimates.}
Since $Q$ is integer-valued,~\eqref{eq:window-left-record} gives, for
$A\leq b\leq B$,
\[
\begin{aligned}
   1
   &\leq Q(b+1)-Q(A)\\
   &=n(b-A+1)-\sum_{j=A}^{b}k_j.
\end{aligned}
\]
Similarly,~\eqref{eq:window-right-record} gives, for $A\leq a\leq B$,
\[
\begin{aligned}
   1
   &\leq Q(B+1)-Q(a)\\
   &=n(B-a+1)-\sum_{j=a}^{B}k_j.
\end{aligned}
\]
Consequently,
\begin{align}
   \sum_{j=A}^{b}k_j+1&\leq n(b-A+1)
      &&(A\leq b\leq B),\label{eq:window-left-boundary}\\
   \sum_{j=a}^{B}k_j+1&\leq n(B-a+1)
      &&(A\leq a\leq B).\label{eq:window-right-boundary}
\end{align}

\textbf{Step 3: Marriage condition.}
Let $G$ be the bipartite graph that joins
$(j,r)\in\mathcal L_J$ to $(t,c)\in K\times\{1,\ldots,n\}$ when
$|j-t|\leq R$, where the time of $(d_*,\star)$ is $d_*$.  For a left
vertex $\ell$, its neighboring times form the interval
\[
   \Gamma_\ell=[\mathrm t(\ell)-R,\mathrm t(\ell)+R]\cap K.
\]
It is nonempty because $\mathrm t(\ell)\in[A,B]$, while the endpoints of
$K$ are $A+R$ and $B-R$.
Fix $\mathcal S\subseteq\mathcal L_J$.  Let $\mathcal P$ be the collection
of maximal consecutive components of
$\bigcup_{\ell\in\mathcal S}\Gamma_\ell$.  Since each $\Gamma_\ell$ is
an integer interval, it is contained in a unique member of $\mathcal P$.
For each $P\in\mathcal P$, let
\[
   \mathcal S_P:=\{\ell\in\mathcal S:\Gamma_\ell\subseteq P\},
\]
and put
\[
   a=\min_{\ell\in\mathcal S_P}\mathrm t(\ell),
   \qquad
   b=\max_{\ell\in\mathcal S_P}\mathrm t(\ell),
   \qquad
   \ell_P=b-a+1.
\]
Then
\[
   P=[a-R,b+R]\cap K
\]
and
\begin{equation}\label{eq:window-component-count}
   |\mathcal S_P|
   \leq\sum_{j=a}^{b}k_j
      +\iota\mathbf 1_{\{d_*\in[a,b]\}}.
\end{equation}
The neighborhood of $\mathcal S_P$ in $G$ is
$P\times\{1,\ldots,n\}$ and has $n|P|$ vertices.

If $P$ contains no endpoint of $K$, then $|P|=\ell_P+2R$, and
\[
   |\mathcal S_P|
   \leq\alpha\ell_P+C+\iota
   <n(\ell_P+2R)=n|P|.
\]
If $P$ contains the left endpoint $A+R$ of $K$ but not the right
endpoint $B-R$, then
$P=[A+R,b+R]$ and, since $\iota\leq1$,
\[
   |\mathcal S_P|
   \leq\sum_{j=A}^{b}k_j+\iota
   \leq\sum_{j=A}^{b}k_j+1
   \leq n|P|
\]
by~\eqref{eq:window-left-boundary}.  The right-endpoint case follows
from~\eqref{eq:window-right-boundary}.  If $P$ contains both endpoints
$A+R$ and $B-R$ of $K$,
then $P=K$ and~\eqref{eq:window-density} together with
\eqref{eq:window-global-slack} gives
\[
   |\mathcal S_P|
   \leq\alpha|J|+C+\iota
   <n(|J|-2R)=n|K|.
\]
The members of $\mathcal P$ are disjoint and the sets $\mathcal S_P$
partition $\mathcal S$.  By the definition of $G$,
\[
   \Gamma_G(\mathcal S)
   =\left(\bigcup_{\ell\in\mathcal S}\Gamma_\ell\right)
      \times\{1,\ldots,n\}
   =\bigsqcup_{P\in\mathcal P}P\times\{1,\ldots,n\}.
\]
Hence
\[
   |\mathcal S|
   =\sum_{P\in\mathcal P}|\mathcal S_P|
   \leq\sum_{P\in\mathcal P} n|P|
   =|\Gamma_G(\mathcal S)|.
\]
By Hall's Marriage Theorem, there exists an $R$-injective map from
$\mathcal L_J$ into $K\times\{1,\ldots,n\}$.
\end{proof}

\subsection{Finite-band perturbations}
\label{sec:finite-band-perturbation}

\begin{lemma}
\label{lem:finite-band-perturbation}
Let $n$ be a positive integer, $\cU\in\Cov(X)$,
$f_0\in C(X,[0,1]^n)$, $\varepsilon>0$, and
$\rho>\dim(X,T)$.  Fix $\lambda,q,\phi,I_\phi$, and $B_0$ in
Lemma~\ref{lem:coding-package}, and let $R$ be a
nonnegative integer.
Set $\widehat M=(\widehat M_{-R},\ldots,\widehat M_R)$, where
\[
   \widehat M_0=B_0,
   \qquad
   \widehat M_r=0\quad(0<|r|\leq R).
\]
If $M=(M_{-R},\ldots,M_R)
\in\bigl(M_{n\times q}(\bbR)\bigr)^{2R+1}$ satisfies
\begin{equation}\label{eq:coefficient-closeness}
   \sum_{r=-R}^{R}\|M_r-\widehat M_r\|_{\max}
   <\min\left\{\frac{\lambda}{2},\frac{\varepsilon}{4}\right\},
\end{equation}
then
\[
   f(x):=\sum_{r=-R}^{R}M_r\phi(T^r x)
\]
defines a member of $C(X,[0,1]^n)$ with
\[
   \|f-f_0\|_\infty<\varepsilon.
\]
Moreover,
\begin{equation}\label{eq:orbit-convolution}
   f(T^k x)=(\cM_M I_\phi(x))_k
   \qquad(x\in X,\ k\in\bbZ).
\end{equation}
\end{lemma}

\begin{proof}
For $z\in(\Deltaq)^{\bbZ}$, put
\[
   h_M(z):=\sum_{r=-R}^{R}M_rz_r.
\]
Since $\|Ap\|_\infty\leq\|A\|_{\max}$ for every $p\in\Deltaq$,
\[
   \|h_M(z)-h_0(z)\|_\infty
   \leq\sum_{r=-R}^{R}\|M_r-\widehat M_r\|_{\max}.
\]
Thus~\eqref{eq:coefficient-closeness} and
\eqref{eq:coding-range-approximation} imply
\[
   h_M((\Deltaq)^{\bbZ})
   \subset[\lambda/2,1-\lambda/2]^n
   \subset[0,1]^n.
\]
We conclude that $f=h_M\circ I_\phi\in C(X,[0,1]^n)$.  Moreover,
\[
\begin{aligned}
   \|f-f_0\|_\infty
   &\leq\|h_M\circ I_\phi-h_0\circ I_\phi\|_\infty
      +\|h_0\circ I_\phi-f_0^\circ\|_\infty
      +\|f_0^\circ-f_0\|_\infty\\
   &<\frac{\varepsilon}{4}+\frac{\varepsilon}{4}+\lambda
   <\varepsilon.
\end{aligned}
\]
\end{proof}

The construction is summarized in Figure~\ref{fig:f-construction}.  In
the qualitative proof the coefficient family is
regular; in the sharp proof it is superregular.
\begin{figure}[htbp]
   \centering
   \includegraphics[width=0.75\textwidth]{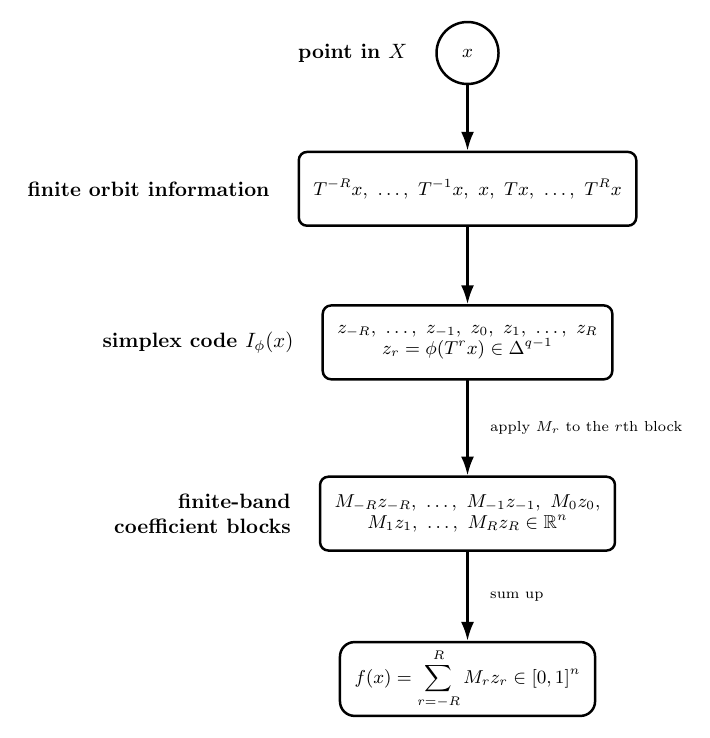}
   \caption{The construction of the perturbed map.}
   \label{fig:f-construction}
\end{figure}

\subsection{The qualitative density argument}
\label{sec:warmup-proof}

The qualitative theorem follows from the next coverwise density
statement, whose numerical hypothesis is $2\rho+2<n$.

\begin{proposition}
\label{prop:warmup-density}
Let $n$ be a positive integer and suppose that there is a number $\rho$
such that
\[
   \dim(X,T)<\rho,
   \qquad
   2\rho+2<n.
\]
Then $\mathcal G(\cU)$ is dense in $C(X,[0,1]^n)$ for every
$\cU\in\Cov(X)$.
\end{proposition}

\begin{proof}
Fix $\cU\in\Cov(X)$, $f_0\in C(X,[0,1]^n)$, and $\varepsilon>0$.
Apply Lemma~\ref{lem:coding-package}.  We obtain
$q,\phi,I_\phi,\Lambda,\kappa,C,B_0$, and $\lambda$ with the properties
listed there.  Set
\[
   \alpha:=2\rho+2<n,
   \qquad
   C_0:=2C,
\]
and choose $R\geq1$ so large that
\[
   2nR>C_0.
\]
By Proposition~\ref{prop:generic-regular}, choose a
regular family $M=(M_{-R},\ldots,M_R)$ satisfying
\eqref{eq:coefficient-closeness}.  Lemma~\ref{lem:finite-band-perturbation}
then gives a map $f\in C(X,[0,1]^n)$ with
$\|f-f_0\|_\infty<\varepsilon$ and
\[
   f(T^k x)=(\cM_M I_\phi(x))_k
   \qquad(x\in X,\ k\in\bbZ).
\]

Suppose, toward a contradiction, that $I_f(x)=I_f(y)$ for points $x,y$
which do not lie in a common member of $\cU$.  Write
\[
   z_j=\phi(T^j x),
   \qquad
   z'_j=\phi(T^j y),
   \qquad
   w_j=z_j-z'_j.
\]
By~\eqref{eq:lambda-separation}, $\Lambda(x)$ and $\Lambda(y)$ are
disjoint.  The sets $\{i:z_{0,i}>0\}$ and $\{i:z'_{0,i}>0\}$ are
nonempty and lie in $\Lambda(x)$ and $\Lambda(y)$, respectively, so
\begin{equation}\label{eq:warmup-nonzero-time}
   w_0\ne0.
\end{equation}
For every $j\in\bbZ$, let
\[
   E_j:=\{i:w_{j,i}\ne0\},
   \qquad
   k_j:=|E_j|.
\]
The sets $\{i:z_{j,i}>0\}$ and $\{i:z'_{j,i}>0\}$ are contained in
$\Lambda(T^j x)$ and $\Lambda(T^j y)$, respectively.  Therefore
\[
   k_j
   \leq|\Lambda(T^j x)|+|\Lambda(T^j y)|
   =\kappa(T^j x)+\kappa(T^j y)+2.
\]
By~\eqref{eq:coding-orbit-bound}, for every finite interval $I$,
\[
   \sum_{j\in I}k_j
   \leq(2\rho+2)|I|+2C
   =\alpha|I|+C_0.
\]

For each $j\in\bbZ$, enumerate $E_j$.  Apply
Lemma~\ref{lem:finite-window-matching} with
$D=\{0\}$ and $\iota=0$.  We obtain an interval $J=[A,B]$
containing zero, its $R$-interior $K=[A+R,B-R]$, and, via these
enumerations of $E_j$, an $R$-injective map from
\[
   E:=\{(j,i):j\in J,\ i\in E_j\}
\]
into $K\times\{1,\ldots,n\}$.  Let $\cR$ be its image.  Then
$|E|=|\cR|$ and the induced map from $E$ to $\cR$ is $R$-bijective.
The regular coefficient family $M$ makes the square matrix
\[
   A_J:=\bigl[(\cM_M\mathbf e_{j,i})_{k,c}\bigr]_{
      (k,c)\in\cR,\ (j,i)\in E}
\]
invertible.

Let $w^{(J)}$ be the sequence equal to $w$ on $J$ and zero outside
$J$.  Then
\begin{equation}\label{eq:warmup-coordinate-expansion}
   w^{(J)}=\sum_{(j,i)\in E}w_{j,i}\mathbf e_{j,i}.
\end{equation}
The equality $I_f(x)=I_f(y)$ and~\eqref{eq:orbit-convolution} give
$\cM_Mw=0$.  If
$k\in K$, then $k+r\in J$ for every $|r|\leq R$, and hence
\[
   (\cM_Mw^{(J)})_k=(\cM_Mw)_k=0.
\]
Applying $\cM_M$ to~\eqref{eq:warmup-coordinate-expansion} and then
restricting the resulting sequence to the coordinates in $\cR$ yields
\[
   A_J(w_{j,i})_{(j,i)\in E}=0.
\]
The vector $(w_{j,i})_{(j,i)\in E}$ is nonzero because $0\in J$ and
\eqref{eq:warmup-nonzero-time} holds.  This contradicts the invertibility
of $A_J$.  Thus $f\in\mathcal G(\cU)$, proving density.
\end{proof}

\begin{proof}[Proof of Theorem~\ref{thm:warmup}]
Put $d=\dim(X,T)<\infty$.  Choose $\rho>d$ and then an integer $n$ with
$2\rho+2<n$.  Proposition~\ref{prop:warmup-density} shows that
$\mathcal G(\cU)$ is dense in $C(X,[0,1]^n)$ for every
$\cU\in\Cov(X)$.  The Baire reduction,
Proposition~\ref{prop:baire-reduction}, completes the proof.
\end{proof}

\begin{proof}[Proof of Corollary~\ref{cor:finite-dim-shift}]
If $\dim(X,T)<\infty$, Theorem~\ref{thm:warmup} supplies an equivariant
embedding into a cubical shift with finite-dimensional alphabet.  Conversely, if
$(X,T)$ embeds into the $n$-cubical shift, monotonicity and the
full-shift formula give
\[
   \dim(X,T)\leq
   \dim\bigl(([0,1]^n)^{\bbZ},\sigma\bigr)=n<\infty
\]
\cite[Proposition~3.1 and Theorem~8.1]{Meyerovitch2026}.
\end{proof}

\section{Signed differences and superregularity}
\label{sec:sharp-upgrade}

The sharp argument uses Lemma~\ref{lem:finite-window-matching} with
$\iota=1$.  In this section, we study signed-difference vectors and
superregular coefficient families in preparation for the proof of the
sharp theorem.

\subsection{Decomposition of simplex differences}
\label{sec:simplex-differences}

\begin{definition}\label{def:signed-difference}
Let $q$ be a positive integer, and let $e_1,\ldots,e_q$ be the standard
basis of $\bbR^q$.  For $1\leq a<b\leq q$ and
$\varepsilon\in\{-1,1\}$, we call
\[
   \varepsilon(e_b-e_a)\in\bbR^q
\]
a \emph{signed-difference vector}.  Its lower index is $a$, its upper
index is $b$, and its sign is $\varepsilon$.  The lower and upper
indices are determined by $a<b$ and do not depend on the sign.

For $j\in\bbZ$, define $\delta(j;a,b)\in(\bbR^q)^{\bbZ}$ by
\[
   \bigl(\delta(j;a,b)\bigr)_\ell
   =
   \begin{cases}
      e_b-e_a,&\ell=j,\\
      0,&\ell\ne j.
   \end{cases}
\]
The sequence $\varepsilon\delta(j;a,b)$ is the \emph{signed difference}
supported at time $j$ associated with $\varepsilon(e_b-e_a)$ and has
the same lower index, upper index, and sign.
\end{definition}

\begin{lemma}
\label{lem:two-simplex-affine}
Let $q$ be a positive integer, let $p,p'\in\Deltaq$, and let
$A,B\subseteq\{1,\ldots,q\}$ be nonempty sets such that
\[
   \{i:p_i>0\}\subseteq A,
   \qquad
   \{i:p'_i>0\}\subseteq B.
\]
\begin{enumerate}[label=\textup{(\arabic*)},leftmargin=2.6em]
\item If $A\cap B\ne\varnothing$, set $S=A\cup B$, $c=\min S$, and
$d=p-p'$.  Then
\[
   p-p'
   =\sum_{\substack{i\in S\setminus\{c\}\\d_i\ne0}}
      |d_i|\,\operatorname{sgn}(d_i)(e_i-e_c).
\]
The sum has at most $(|A|-1)+(|B|-1)$ terms, and its indices
$i$ are pairwise distinct.

\item If $A\cap B=\varnothing$, set $a_0=\min A$ and $b_0=\min B$.  Then
\[
   p-p'=(e_{a_0}-e_{b_0})
      +\sum_{i\in A\setminus\{a_0\}}p_i(e_i-e_{a_0})
      +\sum_{i\in B\setminus\{b_0\}}p'_i(e_{b_0}-e_i).
\]
The two sums have at most
$(|A|-1)+(|B|-1)$ terms with pairwise distinct upper indices.  The
upper index of $e_{a_0}-e_{b_0}$ is $\max\{a_0,b_0\}$ and differs
from all of them.
\end{enumerate}
\end{lemma}

\begin{proof}
First suppose $A\cap B\ne\varnothing$.  Put $S=A\cup B$,
$c=\min S$, and $d=p-p'$.  The vector $d$ is supported in $S$ and its
coordinates sum to zero, so
\[
   d=\sum_{\substack{i\in S\setminus\{c\}\\d_i\ne0}}
      |d_i|\,\operatorname{sgn}(d_i)(e_i-e_c).
\]
The coefficients are positive and the upper indices $i$ are
distinct.  Moreover,
\[
   |S|-1=|A|+|B|-|A\cap B|-1
      \leq(|A|-1)+(|B|-1).
\]

Now suppose $A\cap B=\varnothing$, and put $a_0=\min A$ and
$b_0=\min B$.
Since the coordinates of $p$ sum to one and
$\{i:p_i>0\}\subseteq A$,
\[
   p=e_{a_0}+\sum_{i\in A\setminus\{a_0\}}p_i(e_i-e_{a_0}).
\]
Similarly,
\[
   p'=e_{b_0}+\sum_{i\in B\setminus\{b_0\}}p'_i(e_i-e_{b_0}).
\]
Subtracting gives
\[
   p-p'=(e_{a_0}-e_{b_0})
      +\sum_{i\in A\setminus\{a_0\}}p_i(e_i-e_{a_0})
      +\sum_{i\in B\setminus\{b_0\}}p'_i(e_{b_0}-e_i).
\]
Their upper indices lie in the disjoint union
$(A\setminus\{a_0\})\sqcup(B\setminus\{b_0\})$.  The upper index of the
offset $e_{a_0}-e_{b_0}$ is $\max\{a_0,b_0\}$, which belongs to neither
part of this union.
\end{proof}

\subsection{Signed-difference columns and affine configurations}
\label{sec:affine-configurations}

Fix positive integers $q,n$ and a nonnegative integer $R$, and let
$M=(M_{-R},\ldots,M_R)$ with $M_r\in M_{n\times q}(\bbR)$.  We use the
linear map $\cM_M$ from~\eqref{eq:convolution}.

\begin{definition}\label{def:upper-compatible}
Let $q$ be a positive integer.  A finite set
$\cE\subset(\bbR^q)^{\bbZ}$ of signed differences is
\emph{upper-compatible} if distinct vectors
$\varepsilon\delta(j;a,b)$ and
$\varepsilon'\delta(j';a',b')$ satisfy
\[
   j=j'\quad\Longrightarrow\quad b\ne b'.
\]
Thus upper indices are pairwise distinct among vectors supported at
the same time; lower indices may repeat.
\end{definition}

We use Definition~\ref{def:r-injective} for signed differences by
identifying
\[
   \varepsilon\delta(j;a,b)
   \longleftrightarrow
   (j,(a,b,\varepsilon)).
\]
For a finite set of signed differences and a finite set
$\cR\subset\bbZ\times\{1,\ldots,n\}$, the associated time-incidence
graph joins a vector supported at time $j$ to $(k,c)\in\cR$ exactly when
$|j-k|\leq R$.
By the definition of $\cM_M$, this graph records precisely the structural
zero pattern of the corresponding convolution columns.  We order
$\cR$ lexicographically by $(k,c)$ and signed differences
lexicographically by $(j,a,b,\varepsilon)$, with $-1<1$.

\begin{definition}
\label{def:admissible-affine-configuration}
Let $q,n$ be positive integers and $R$ a nonnegative integer.  An
\emph{admissible affine configuration} is a tuple
\[
   \mathfrak a=(\cE,D,(v_j)_{j\in D},\cR)
\]
with the following properties.
\begin{enumerate}[label=\textup{(\roman*)},leftmargin=2.5em]
\item $\cE$ is a finite upper-compatible set of signed differences.
\item $D\subset\bbZ$ is finite and nonempty, and for every $j\in D$,
\[
   v_j=\varepsilon_j\delta(j;a_j,b_j),
   \qquad
   \varepsilon_j\in\{-1,1\},
   \quad1\leq a_j<b_j\leq q,
\]
is a signed difference whose upper index $b_j$ differs from the upper index of
every vector in $\cE$ supported at time $j$.  In particular,
$v_j\notin\cE$.
\item $\cR\subset\bbZ\times\{1,\ldots,n\}$ is finite with
$|\cR|=|\cE|+1$, and, for $j_*:=\min D$, there is an $R$-bijective map
from $\cE\cup\{v_{j_*}\}$ to $\cR$.
\end{enumerate}
\end{definition}

For such a configuration, set
\[
   V_D:=\sum_{j\in D}v_j.
\]
The notation
\[
   \left[(\cM_Mu)_{u\in\cE},\ \cM_MV_D\right]_{\cR}
\]
denotes the square matrix whose columns indexed by $u\in\cE$ are the
restrictions of $\cM_Mu$ to $\cR$, and whose final column is the
restriction of $\cM_MV_D$.

\subsection{Superregular coefficient families}
\label{sec:superregularity}

The required coefficient families are generic because every determinant
arising from an admissible affine configuration is a nonzero polynomial.

\begin{lemma}
\label{lem:affine-minor}
Let $q,n$ be positive integers and $R$ a nonnegative integer, and let
$\mathfrak a=(\cE,D,(v_j)_{j\in D},\cR)$ be an admissible affine
configuration.  For
$M\in\bigl(M_{n\times q}(\bbR)\bigr)^{2R+1}$, set
\[
   P_{\mathfrak a}(M)
   :=\det\left[
      (\cM_M u)_{u\in\cE},\ \cM_M V_D
     \right]_{\cR}.
\]
Then $P_{\mathfrak a}$ is a nonzero polynomial on
$\bigl(M_{n\times q}(\bbR)\bigr)^{2R+1}$.
\end{lemma}

\begin{proof}
Write
\[
   N:=|\cR|=|\cE|+1,\qquad
   \cE=\{u_1,\ldots,u_{N-1}\},\qquad
   \cR=\{\rho_1,\ldots,\rho_N\}
\]
in the fixed orders, and write
\[
   \rho_\ell=(t_\ell,\gamma_\ell),\qquad
   u_s=\varepsilon_s\delta(\tau_s;a_s,b_s)
   \quad(s<N).
\]
Choose an integer $L>q+n$.  For $|r|\leq R$,
$1\leq c\leq n$, and $1\leq i\leq q$, put
\[
   S_{r,c,i}:=(-Lr+q+c-i)^2
\]
and set
\begin{equation}
   (\widetilde M_r)_{c,i}
   :=X_1^{\,r+R}X_2^{\,q-i}X_3^{\,S_{r,c,i}}
   \in\bbZ[X_1,X_2,X_3].\label{eq:affine-specialization}
\end{equation}
Observe that all exponents in~\eqref{eq:affine-specialization} are
nonnegative.  This
assignment of the matrix-entry variables extends to a ring homomorphism into
$\bbZ[X_1,X_2,X_3]$.  Let
\[
   Q_{\mathfrak a}(X_1,X_2,X_3):=P_{\mathfrak a}(\widetilde M)
\]
be the image of $P_{\mathfrak a}$ under~\eqref{eq:affine-specialization}.  It is
enough to prove that $Q_{\mathfrak a}$ is nonzero.

For $j\in D$, write
\[
   v_j=\varepsilon_N\delta(\tau_N;a_N,b_N),
   \qquad \tau_N=j,
\]
and set $u_N:=v_j$.
Each term in the expansion of $Q_{\mathfrak a}$ is indexed by
\[
   \omega=(j,\sigma,\boldsymbol i),\qquad
   \sigma\in\mathfrak S_N,\qquad
   \boldsymbol i=(i_1,\ldots,i_N)
      \in\prod_{s=1}^{N}\{a_s,b_s\}.
\]
Here $j$ selects the summand $v_j$ in the last column, $\sigma$ assigns
column $s$ to the output row $\rho_{\sigma(s)}$ in the determinant
expansion, and $i_s=a_s$ or $i_s=b_s$ selects the lower- or upper-index
entry, respectively, of the signed difference $u_s$.  Observe that the
term indexed by $\omega$ vanishes unless
\begin{equation}
   |\tau_s-t_{\sigma(s)}|\leq R
   \qquad\text{for all }1\leq s\leq N.
   \label{eq:term-band-condition}
\end{equation}
When~\eqref{eq:term-band-condition} holds, the definition of
$\cM_M$ and~\eqref{eq:affine-specialization} show that the term indexed by
$\omega$ contributes
\[
   a_\omega X_1^{A_1(\omega)}X_2^{A_2(\omega)}X_3^{A_3(\omega)},
   \qquad
   a_\omega:=\operatorname{sgn}(\sigma)
       (-1)^{\#\{s:i_s=a_s\}}\prod_{s=1}^{N}\varepsilon_s
       \in\{-1,1\},
\]
where
\begin{equation}
\begin{aligned}
   A_1(\omega)
      &:=\sum_{s=1}^{N}\bigl(\tau_s-t_{\sigma(s)}+R\bigr),\\
   A_2(\omega)
      &:=\sum_{s=1}^{N}(q-i_s),\\
   A_3(\omega)
      &:=\sum_{s=1}^{N}
         S_{\tau_s-t_{\sigma(s)},\gamma_{\sigma(s)},i_s}.
\end{aligned}\label{eq:affine-exponents}
\end{equation}
Set $j_*:=\min D$.  Condition~\textup{(iii)} in
Definition~\ref{def:admissible-affine-configuration} supplies an
$R$-bijective map from $\cE\cup\{v_{j_*}\}$ to $\cR$, which determines
a permutation $\sigma\in\mathfrak S_N$.  Hence there exists
\[
   \omega=(j_*,\sigma,\boldsymbol i),
   \qquad
   \boldsymbol i\in\prod_{s=1}^{N}\{a_s,b_s\},
\]
satisfying~\eqref{eq:term-band-condition}.

We order the monomials indexed by $\omega$ satisfying
\eqref{eq:term-band-condition} as follows.  First compare their
$X_1$-exponents $A_1(\omega)$.  If these are equal, compare their
$X_2$-exponents $A_2(\omega)$.  If both are equal, compare their
$X_3$-exponents $A_3(\omega)$.  In each comparison, the smaller exponent
takes precedence.

\textbf{Step 1: Minimizing $A_1(\omega)$ forces $j=j_*$.}
Because $\sigma$ permutes all rows of the fixed set $\cR$,
\[
\begin{aligned}
   A_1(\omega)
      &=NR+\sum_{s=1}^{N}\tau_s
           -\sum_{s=1}^{N}t_{\sigma(s)}\\
      &=NR+\sum_{s=1}^{N-1}\tau_s+j
           -\sum_{\ell=1}^{N}t_\ell.
\end{aligned}
\]
Thus $A_1(\omega)$ depends only on $j$ and is strictly increasing in
$j$.  Since an $\omega$ satisfying~\eqref{eq:term-band-condition} with
$j=j_*$ exists, $A_1(\omega)$ is minimized exactly when $j=j_*$. 

\textbf{Step 2: Minimizing $A_2(\omega)$ forces $i_s=b_s$ for all
$1\leq s\leq N$.}
After Step~1, restrict to $j=j_*$.  Since
\eqref{eq:term-band-condition} does not involve $\boldsymbol i$,
replacing $i_s=a_s$ by $i_s=b_s$ preserves
\eqref{eq:term-band-condition}.  Moreover,
\[
   (q-a_s)-(q-b_s)=b_s-a_s>0,
\]
so this replacement strictly decreases $A_2(\omega)$.  Therefore,
among the indices minimizing $A_1$, the value $A_2(\omega)$ is minimized
exactly when
\[
   i_s=b_s\qquad(1\leq s\leq N).
\]
Thus $\boldsymbol i$ is uniquely determined, and only $\sigma$ remains
to be chosen.

\textbf{Step 3: Minimizing $A_3(\omega)$ determines the row assignment
uniquely.}
After Steps 1 and~2, $j=j_*$ and $i_s=b_s$ for every $s$.  The input
labels
\[
   (\tau_s,b_s)\qquad(1\leq s\leq N)
\]
are pairwise distinct: this follows from upper compatibility for the
members of $\cE$ and from condition~\textup{(ii)} of
Definition~\ref{def:admissible-affine-configuration} for $v_{j_*}$.
The output labels $(t_\ell,\gamma_\ell)$ are pairwise distinct because
the elements $\rho_\ell$ of $\cR$ are distinct.

For an input label $(j,i)$ and an output label $(k,c)$, set
\[
   \xi_{j,i}:=Lj+i,\qquad
   \eta_{k,c}:=Lk+q+c.
\]
Since $L>q+n$,
\[
\begin{aligned}
   \xi_{j,i}<\xi_{j',i'}
   &\iff j<j'\ \text{or}\ (j=j'\ \text{and}\ i<i'),\\
   \eta_{k,c}<\eta_{k',c'}
   &\iff k<k'\ \text{or}\ (k=k'\ \text{and}\ c<c').
\end{aligned}
\]
Moreover,
\begin{equation}
   S_{j-k,c,i}
      =\bigl(-L(j-k)+q+c-i\bigr)^2
      =\bigl(\eta_{k,c}-\xi_{j,i}\bigr)^2.\label{eq:affine-distance}
\end{equation}
Arrange the numbers $\xi_{\tau_s,b_s}$ and
$\eta_{t_\ell,\gamma_\ell}$ in increasing order and denote them by
\[
   \xi_1<\cdots<\xi_N,
   \qquad
   \eta_1<\cdots<\eta_N,
\]
respectively.  Denote their corresponding time coordinates by $j_p$ and $k_p$,
respectively.  Thus both $(j_p)_{p=1}^N$ and $(k_p)_{p=1}^N$ are
nondecreasing.  By~\eqref{eq:affine-exponents} and
\eqref{eq:affine-distance}, minimizing $A_3(\omega)$ over the
permutations $\sigma$ satisfying~\eqref{eq:term-band-condition} is
equivalent to minimizing
\begin{equation}
   \sum_{p=1}^{N}\bigl(\eta_{\pi(p)}-\xi_p\bigr)^2\label{eq:affine-cost}
\end{equation}
over all $R$-bijective maps
\[
   (j_p,\xi_p)\longmapsto(k_{\pi(p)},\eta_{\pi(p)}).
\]

Condition~\textup{(iii)} supplies at least one such map.  Hence
Lemma~\ref{lem:banded-uncrossing} shows that the order-preserving map
\[
   (j_p,\xi_p)\longmapsto(k_p,\eta_p)
   \qquad(1\leq p\leq N)
\]
is $R$-bijective and is the unique $R$-bijective minimizer of
\eqref{eq:affine-cost}.

Because the input labels $(\tau_s,b_s)$, $1\leq s\leq N$, are pairwise
distinct, each one identifies exactly one column $u_s$.
Therefore, the order-preserving map uniquely determines $\sigma$.  Let
$\omega_0$ be the resulting index.  By Steps 1--3, the term indexed by
$\omega_0$ is the unique term contributing the smallest monomial in the
order defined above.  Consequently,
the coefficient of
\[
   X_1^{A_1(\omega_0)}X_2^{A_2(\omega_0)}X_3^{A_3(\omega_0)}
\]
in $Q_{\mathfrak a}$ is exactly $a_{\omega_0}\in\{-1,1\}$.  Hence this
least monomial cannot cancel, and
\[
   Q_{\mathfrak a}(X_1,X_2,X_3)\ne0.
\]
If $P_{\mathfrak a}$ were the zero polynomial in the entries of
$M_{-R},\ldots,M_R$, every polynomial specialization would vanish.
Hence $P_{\mathfrak a}$ is nonzero.
\end{proof}

\begin{definition}
\label{def:superregular}
Let $q,n$ be positive integers and $R$ a nonnegative integer.  A
coefficient family
\[
   M=(M_{-R},\ldots,M_R)
   \in\bigl(M_{n\times q}(\bbR)\bigr)^{2R+1}
\]
is \emph{superregular} if, for every admissible affine
configuration $\mathfrak a=(\cE,D,(v_j)_{j\in D},\cR)$, the square matrix
\[
   \left[(\cM_M u)_{u\in\cE},\ \cM_M V_D\right]_{\cR}
\]
is invertible.
\end{definition}

\begin{corollary}
\label{cor:generic-superregular}
For positive integers $q,n$ and a nonnegative integer $R$,
superregular coefficient families form a dense $G_\delta$ subset of
\[
   \bigl(M_{n\times q}(\bbR)\bigr)^{2R+1}.
\]
\end{corollary}

\begin{proof}
Let
\[
   \mathcal L_q
   :=\{(a,b,\varepsilon):1\leq a<b\leq q,
      \ \varepsilon\in\{-1,1\}\},
\]
and let $\operatorname{Fin}(S)$ denote the set of finite subsets of $S$.
The set of admissible affine configurations is contained in
\[
\operatorname{Fin}(\bbZ\times\mathcal L_q)
\times
\bigcup_{D\in\operatorname{Fin}(\bbZ)}
   \bigl(\{D\}\times\mathcal L_q^D\bigr)
\times
\operatorname{Fin}\bigl(\bbZ\times\{1,\ldots,n\}\bigr).
\]
The set on the right is countable.  By
Lemma~\ref{lem:affine-minor}, the determinant associated with each
admissible affine configuration is a nonzero polynomial.  The complement of its
zero set is open; it is dense because the zero set of a nonzero real
polynomial has empty interior.  The countable intersection of these open
dense sets is a dense $G_\delta$ by the Baire category theorem, and its
members are exactly the superregular families.
\end{proof}

\section{Proof of the sharp embedding theorem}
\label{sec:sharp-proof}

It remains to verify the density hypothesis in
Proposition~\ref{prop:baire-reduction} under the sharp dimension bound.
The proof combines the coding lemma with the superregular
coefficient families supplied by
Corollary~\ref{cor:generic-superregular}.

\begin{proposition}\label{prop:sharp-density}
Let $n$ be a positive integer.  Assume
\[
   d:=\dim(X,T)<\frac n2.
\]
Then $\mathcal G(\cU)$ is dense in $C(X,[0,1]^n)$ for every
$\cU\in\Cov(X)$.
\end{proposition}

\begin{proof}
Fix $\cU\in\Cov(X)$, $f_0\in C(X,[0,1]^n)$, and $\varepsilon>0$.
Choose
\[
   d<\rho<\frac n2
\]
and apply Lemma~\ref{lem:coding-package}.  This gives
$q,\phi,I_\phi,\Lambda,\kappa,C,B_0$, and $\lambda$.  Put
\[
   \alpha:=2\rho<n,
   \qquad
   C_0:=2C.
\]
Choose $R\geq1$ so large that
\begin{equation}
   2nR>C_0+1.\label{eq:5.24}
\end{equation}
By Corollary~\ref{cor:generic-superregular}, choose a superregular
coefficient family $M=(M_{-R},\ldots,M_R)$ satisfying
\eqref{eq:coefficient-closeness}.  Lemma~\ref{lem:finite-band-perturbation}
produces $f\in C(X,[0,1]^n)$ with
$\|f-f_0\|_\infty<\varepsilon$ and
\[
   f(T^k x)=(\cM_M I_\phi(x))_k
   \qquad(x\in X,\ k\in\bbZ).
\]

It remains to prove that $f\in\mathcal G(\cU)$.  We argue by
contradiction.
Suppose, toward a contradiction, that
\begin{equation}
   I_f(x)=I_f(y)\label{eq:5.33}
\end{equation}
and assume that no $U\in\cU$ contains both $x$ and $y$.

\medskip
\noindent\textbf{Step 1: Decomposition into signed-difference vectors.}

Write the three bi-infinite sequences
\[
\begin{aligned}
   z=I_\phi(x)&=(z_j)_{j\in\bbZ},
      &\text{where}\qquad z_j&=\phi(T^j x),\\
   z'=I_\phi(y)&=(z'_j)_{j\in\bbZ},
      &\text{where}\qquad z'_j&=\phi(T^j y),\\
   w=z-z'&=(w_j)_{j\in\bbZ},
      &\text{where}\qquad w_j&=z_j-z'_j,
\end{aligned}
\]
Thus $z,z'\in(\Deltaq)^{\bbZ}$ and $w\in(\bbR^q)^{\bbZ}$.
We apply Lemma~\ref{lem:two-simplex-affine} at each $j\in\bbZ$ to
decompose $w_j=\phi(T^j x)-\phi(T^j y)$ into signed differences.
By Lemma~\ref{lem:coding-package},
\begin{equation}
   \Lambda(x)\cap\Lambda(y)=\varnothing.\label{eq:5.36}
\end{equation}

Define
\begin{equation}
   D=\{j\in\bbZ:\Lambda(T^j x)\cap\Lambda(T^j y)=\varnothing\}.\label{eq:5.37}
\end{equation}
By \eqref{eq:5.36}, $0\in D$, so $D$ is nonempty.

The sets $\Lambda(T^j x)$ and $\Lambda(T^j y)$ are nonempty, and
Lemma~\ref{lem:coding-package} gives
$\{i:z_{j,i}>0\}\subseteq\Lambda(T^j x)$ and
$\{i:z'_{j,i}>0\}\subseteq\Lambda(T^j y)$.  Hence
Lemma~\ref{lem:two-simplex-affine} applies at every time $j$.  We treat $j\notin D$
and $j\in D$ separately.

\textbf{Case 1: $j\notin D$.}
Set
\[
   S_j=\Lambda(T^j x)\cup\Lambda(T^j y),
   \qquad
   c_j=\min S_j.
\]
Applying Lemma~\ref{lem:two-simplex-affine} with
$p=z_j$, $p'=z'_j$, $A=\Lambda(T^j x)$, and
$B=\Lambda(T^j y)$ gives
\begin{equation}
   w_j=\sum_{i\in S_j\setminus\{c_j\}}
       w_{j,i}(e_i-e_{c_j}).\label{eq:5.38}
\end{equation}
Define
\[
   \mathcal T_j
   =\left\{
      \operatorname{sgn}(w_{j,i})\delta(j;c_j,i):
      i\in S_j\setminus\{c_j\},\ w_{j,i}\ne0
     \right\}.
\]
Then
\begin{equation}
   |\mathcal T_j|\leq |S_j|-1
   \leq(|\Lambda(T^j x)|-1)+(|\Lambda(T^j y)|-1).\label{eq:5.39}
\end{equation}
The last inequality uses
$|\Lambda(T^j x)\cap\Lambda(T^j y)|\geq1$.
For the vector
$\theta=\operatorname{sgn}(w_{j,i})\delta(j;c_j,i)$ indexed by $i$,
set $\beta_\theta=|w_{j,i}|$.

\textbf{Case 2: $j\in D$.}
Let
\[
   a_j=\min\Lambda(T^j x),
   \qquad
   b_j=\min\Lambda(T^j y).
\]
The sets $\Lambda(T^j x)$ and $\Lambda(T^j y)$ are disjoint, so $a_j\ne b_j$.  Define
\[
   v_j=
   \begin{cases}
      -\delta(j;a_j,b_j), & a_j<b_j,\\
      \phantom{-}\delta(j;b_j,a_j), & b_j<a_j.
   \end{cases}
\]
Its value at time $j$ is $e_{a_j}-e_{b_j}$.  Applying
Lemma~\ref{lem:two-simplex-affine} with $p=z_j$, $p'=z'_j$,
$A=\Lambda(T^j x)$, and $B=\Lambda(T^j y)$ gives
\begin{equation}
 w_j
 = (e_{a_j}-e_{b_j})
   +\sum_{i\in\Lambda(T^j x)\setminus\{a_j\}} z_{j,i}(e_i-e_{a_j})
   -\sum_{i\in\Lambda(T^j y)\setminus\{b_j\}} z'_{j,i}(e_i-e_{b_j}).\label{eq:5.40}
\end{equation}
Define
\[
\begin{aligned}
   \mathcal T_j
   ={}&\left\{
      \delta(j;a_j,i):
      i\in\Lambda(T^j x)\setminus\{a_j\},\ z_{j,i}>0
      \right\}\\
   &{}\cup\left\{
      -\delta(j;b_j,i):
      i\in\Lambda(T^j y)\setminus\{b_j\},\ z'_{j,i}>0
      \right\}.
\end{aligned}
\]
Then
\begin{equation}
   |\mathcal T_j|
   \leq(|\Lambda(T^j x)|-1)+(|\Lambda(T^j y)|-1).\label{eq:5.41}
\end{equation}
For the vector $\theta=\delta(j;a_j,i)$ indexed by $i$, set
$\beta_\theta=z_{j,i}$; for the vector
$\theta=-\delta(j;b_j,i)$ indexed by $i$, set
$\beta_\theta=z'_{j,i}$.
In both cases every $\beta_\theta$ is positive.

In Case~1,
$c_j<i$ for every $i\in S_j\setminus\{c_j\}$; in Case~2,
$a_j<i$ for every $i\in\Lambda(T^j x)\setminus\{a_j\}$ and
$b_j<i$ for every $i\in\Lambda(T^j y)\setminus\{b_j\}$.  Hence every
$\theta\in\mathcal T_j$ has the form
\[
   \theta=\varepsilon\delta(j;a,b),
   \qquad 1\leq a<b\leq q,
   \quad \varepsilon\in\{-1,1\},
\]
and is therefore a signed difference.  Set $k_j=|\mathcal T_j|$.  Every
member of $\mathcal T_j$ is supported at time $j$.  Let $\widehat w^{(j)}$ be the
sequence supported at time $j$ with value $w_j$
there.  The decompositions \eqref{eq:5.38} and \eqref{eq:5.40} then give
\[
   \widehat w^{(j)}
   =
   \begin{cases}
      \displaystyle\sum_{\theta\in\mathcal T_j}\beta_\theta\theta,
         & j\notin D,\\[1.2ex]
      \displaystyle\sum_{\theta\in\mathcal T_j}\beta_\theta\theta+v_j,
         & j\in D.
   \end{cases}
\]
For every $j\in\bbZ$, the upper indices of the members of $\mathcal T_j$
are pairwise distinct: they lie in $S_j\setminus\{c_j\}$ when $j\notin D$, and in
the disjoint union
\[
   \bigl(\Lambda(T^j x)\setminus\{a_j\}\bigr)
   \mathbin{\sqcup}
   \bigl(\Lambda(T^j y)\setminus\{b_j\}\bigr)
\]
when $j\in D$.  In the latter case, the upper index of $v_j$ is
$\max\{a_j,b_j\}$ and does not belong to this union.  Consequently,
every finite subset of $\bigcup_{j\in\bbZ}\mathcal T_j$ is
upper-compatible.
Moreover, at every $j\in D$, the upper index of $v_j$ differs from all
upper indices of members of $\mathcal T_j$.

Finally, \eqref{eq:5.39} and \eqref{eq:5.41}, together with
$\kappa(T^j x)=|\Lambda(T^j x)|-1$ and
$\kappa(T^j y)=|\Lambda(T^j y)|-1$, give
\begin{equation}
   k_j\leq\kappa(T^j x)+\kappa(T^j y).\label{eq:5.42}
\end{equation}
For every finite integer interval $I$, summing \eqref{eq:5.42} and
applying \eqref{eq:coding-orbit-bound} separately to $x$ and $y$ gives
\begin{equation}
\begin{aligned}
   \sum_{j\in I}k_j
   &\leq\sum_{j\in I}\kappa(T^j x)
       +\sum_{j\in I}\kappa(T^j y)\\
   &\leq2\rho|I|+2C\\
   &=\alpha|I|+C_0.
\end{aligned}\label{eq:5.43}
\end{equation}

\medskip
\noindent\textbf{Step 2: Construction of an admissible affine
configuration.}

For every $j\in\bbZ$, choose a labeling
\[
   \mathcal T_j=\{\theta_{j,r}:1\leq r\leq k_j\}.
\]
Applying Lemma~\ref{lem:finite-window-matching} to $(k_j)$, $D$,
$\iota=1$, $\alpha$, $C_0$, and the bandwidth $R$, with its hypotheses
verified by \eqref{eq:5.43}, \eqref{eq:5.37}, and \eqref{eq:5.24}, we obtain
\[
   J=[A,B],
   \qquad
   d_*=\min(D\cap J),
   \qquad
   K=[A+R,B-R],
\]
and an $R$-injective map
\[
   \pi\colon\mathcal L_J\longrightarrow
      K\times\{1,\ldots,n\}
\]
where
\[
   \mathcal L_J
   :=\{(j,r):j\in J,\ 1\leq r\leq k_j\}
      \mathbin{\sqcup}\{(d_*,\star)\}.
\]

Let
\[
   \cE_J=\bigcup_{j\in J}\mathcal T_j
\]
be the finite set of signed differences selected over $J$.  Vectors belonging to
different $\mathcal T_j$ have different support times, so this union is
disjoint and $|\cE_J|=\sum_{j\in J}k_j$.  Set
\[
   V_J=\sum_{j\in D\cap J}v_j.
\]
The upper-index observation above implies that no vector of
$\mathcal T_{d_*}$ equals $v_{d_*}$, and vectors of $\mathcal T_j$ with
$j\ne d_*$ have a different support time; hence
$v_{d_*}\notin\cE_J$.  Hence the map
\[
   \mathcal L_J\longrightarrow\cE_J\cup\{v_{d_*}\},
   \qquad
   (j,r)\longmapsto\theta_{j,r},
   \quad
   (d_*,\star)\longmapsto v_{d_*},
\]
is bijective.  Each image has the same support time as its preimage, so
$\pi$ induces an $R$-injective map
\[
   \widetilde\pi\colon\cE_J\cup\{v_{d_*}\}
      \longrightarrow K\times\{1,\ldots,n\},
   \qquad
   \widetilde\pi(\theta_{j,r})=\pi(j,r),
   \quad
   \widetilde\pi(v_{d_*})=\pi(d_*,\star).
\]
Let $\cR$ be the image of $\widetilde\pi$.  Then
$|\cR|=|\cE_J|+1$, and $\widetilde\pi$ is $R$-bijective from
$\cE_J\cup\{v_{d_*}\}$ to $\cR$.

We verify the three conditions of
Definition~\ref{def:admissible-affine-configuration} for
\[
   \mathfrak a_J
   :=\bigl(\cE_J,D\cap J,(v_j)_{j\in D\cap J},\cR\bigr).
\]
Condition \textup{(i)} holds because $\cE_J$ is finite and
upper-compatible,
as shown above.  For condition \textup{(ii)}, $D\cap J$ is finite
and nonempty with minimum $d_*$.  The preceding decomposition also
shows that, at every $j\in D\cap J$, the upper index of $v_j$ differs from
all upper indices of members of $\mathcal T_j$.  For condition \textup{(iii)}, one has
$|\cR|=|\cE_J|+1$, and, since $j_*:=\min(D\cap J)=d_*$, the map
$\widetilde\pi$ constructed above is $R$-bijective from
$\cE_J\cup\{v_{d_*}\}$ to $\cR$.  Thus $\mathfrak a_J$ is an admissible
affine configuration.
Superregularity of $M$ implies that the following square matrix
is invertible:
\[
   A_J
   :=\left[(\cM_M\theta)_{\theta\in\cE_J},\ \cM_M V_J\right]_{\cR}
\]

\medskip
\noindent\textbf{Step 3: Contradiction with superregularity.}

On the other hand, \eqref{eq:5.33} and \eqref{eq:orbit-convolution} give,
for every $k\in\bbZ$,
\begin{equation}
\begin{aligned}
   0
   &=f(T^k x)-f(T^k y)\\
   &=(\cM_Mz)_k-(\cM_Mz')_k
    =(\cM_M w)_k.
\end{aligned}\label{eq:5.47}
\end{equation}
Define the sequence $w^{(J)}\in(\bbR^q)^{\bbZ}$ by
\[
   w^{(J)}_j=
   \begin{cases}
      w_j, & j\in J,\\
      0,   & j\notin J.
   \end{cases}
\]
It follows from Step~1 that
\[
   w^{(J)}
   =\sum_{j\in J}\widehat w^{(j)}
   =\sum_{\theta\in\cE_J}\beta_\theta\theta+V_J.
\]
Fix $k\in K$.  Since $K=[A+R,B-R]$ and $J=[A,B]$, one has
$k+r\in J$ for every $-R\leq r\leq R$.  Therefore
$w^{(J)}_{k+r}=w_{k+r}$ for every $-R\leq r\leq R$, and
\[
\begin{aligned}
   (\cM_Mw^{(J)})_k
   &=\sum_{r=-R}^{R}M_r w^{(J)}_{k+r}\\
   &=\sum_{r=-R}^{R}M_r w_{k+r}\\
   &=(\cM_Mw)_k=0,
\end{aligned}
\]
where the last equality is \eqref{eq:5.47}.  Thus
$\cM_Mw^{(J)}$ vanishes at every output time in $K$.

Apply $\cM_M$ to the decomposition of $w^{(J)}$.  Since
$\cR\subset K\times\{1,\ldots,n\}$ and $\cM_Mw^{(J)}$ vanishes on $K$,
restriction to the coordinates in $\cR$ gives
\[
   A_J\bigl((\beta_\theta)_{\theta\in\cE_J},1\bigr)^\mathsf T=0.
\]
The coefficient vector
$\bigl((\beta_\theta)_{\theta\in\cE_J},1\bigr)^\mathsf T$ is nonzero because
its last coordinate is $1$, contradicting the invertibility of $A_J$.
Thus no pair outside a common member of $\cU$ can have the same image
under $I_f$.  Therefore $f\in\mathcal G(\cU)$.

Since $f_0$ and $\varepsilon$ were arbitrary, $\mathcal G(\cU)$ is dense.

\end{proof}

\begin{proof}[Proof of Theorem~\ref{thm:main}]
Proposition~\ref{prop:sharp-density} verifies the density hypothesis in
Proposition~\ref{prop:baire-reduction}; the latter gives the dense
$G_\delta$ conclusion and the equivariant topological embedding.
\end{proof}

\section{Applications}
\label{sec:applications}

We first specialize Theorem~\ref{thm:main} to the small boundary and
marker settings.  We then compare it with finite-dimensional embedding
theorems and isolate the role of periodic points.

\subsection{The small boundary property}

Recall that $(X,T)$ has the \emph{small boundary property} if it has a
basis consisting of open sets $U$ such that
\[
   \mu(\partial U)=0
   \qquad\text{for every }\mu\in\Prob(X,T).
\]
Meyerovitch proved that this property is exactly the zero-dimensional
case of dynamical dimension
\cite[Theorem~5.1]{Meyerovitch2026}.

\begin{corollary}\label{cor:sbp-embedding}
If $(X,T)$ has the small boundary property, then
\[
   \bigl\{f\in C(X,[0,1]): I_f\text{ is a topological embedding}\bigr\}
\]
is a dense $G_\delta$ subset of $C(X,[0,1])$.  In particular,
$(X,T)$ embeds equivariantly into the one-dimensional cubical shift.
\end{corollary}

\begin{proof}
The small boundary property gives $\dim(X,T)=0$ by
\cite[Theorem~5.1]{Meyerovitch2026}.  Since $0<1/2$, the assertion is
Theorem~\ref{thm:main} with $n=1$.
\end{proof}

\subsection{The marker property}

Recall that $(X,T)$ has the \emph{marker property} if, for every positive
integer $N$, there is an open set $U\subset X$ such that
\[
   U\cap T^kU=\varnothing\quad(1\leq k\leq N),
   \qquad
   X=\bigcup_{k\in\bbZ}T^kU.
\]
\begin{corollary}\label{cor:marker-embedding}
Let $(X,T)$ have the marker property, and let $n$ be a positive integer.  If
\[
   \operatorname{mdim}(X,T)<\frac n2,
\]
then
\[
   \bigl\{f\in C(X,[0,1]^n): I_f\text{ is a topological embedding}\bigr\}
\]
is a dense $G_\delta$ subset of $C(X,[0,1]^n)$.
\end{corollary}

\begin{proof}
Under the marker property,
\[
   \dim(X,T)=\operatorname{mdim}(X,T)
\]
by \cite[discussion preceding Lemma~9.1 and Theorem~9.2]{Meyerovitch2026}.
The assumed inequality therefore becomes
$\dim(X,T)<n/2$, and Theorem~\ref{thm:main} applies.
\end{proof}

Corollary~\ref{cor:marker-embedding} recovers the embedding theorem of
Gutman and Tsukamoto \cite[Theorem~1.4]{GutmanTsukamoto2020}.

\subsection{Finite-dimensional systems}

Dynamical dimension never exceeds the ordinary covering dimension of
the phase space.  Combining the following comparison with
Theorem~\ref{thm:main} yields a dynamical analogue of the
Menger--N\"obeling theorem.

\begin{proposition}\label{prop:dydim-coverdim}
For every topological dynamical system $(X,T)$,
\[
   \dim(X,T)\leq\dim X.
\]
\end{proposition}

\begin{proof}
There is nothing to prove if $\dim X=\infty$.  Suppose that
$\dim X=d<\infty$, and let $\cU$ be a finite open cover of $X$.
By the definition of covering dimension, $\cU$ has a finite open
refinement $\cV$ satisfying
\[
   \sup_{x\in X}\ord(\cV,x)\leq d.
\]
For every $T$-invariant Borel probability measure $\mu$,
\[
   \int_X\ord(\cV,x)\,d\mu(x)\leq d.
\]
Hence $\ord(\cV,T)\leq d$, so $\dim(\cU,T)\leq d$.  Taking the
supremum over $\cU$ gives $\dim(X,T)\leq d$.
\end{proof}

\begin{corollary}\label{cor:finite-dimensional-embedding}
Let $(X,T)$ be a topological dynamical system with $\dim X=d<\infty$.
For every integer $n\geq2d+1$, the set
\[
   \bigl\{f\in C(X,[0,1]^n): I_f\text{ is a topological embedding}\bigr\}
\]
is a dense $G_\delta$ subset of $C(X,[0,1]^n)$.
\end{corollary}

\begin{proof}
Proposition~\ref{prop:dydim-coverdim} gives
$\dim(X,T)\leq d<n/2$.  Apply Theorem~\ref{thm:main}.
\end{proof}

For the identity action, the existence conclusion in
Corollary~\ref{cor:finite-dimensional-embedding} reduces to the classical
$2d+1$ embedding theorem.  For a general homeomorphism it says that the
same universal bound remains valid for equivariant orbit embeddings.
Proposition~\ref{prop:sharp} below shows that the coefficient two cannot
be improved uniformly.

\section{Sharpness of the strict one-half threshold}
\label{sec:sharpness}

We now show that neither the constant $1/2$ nor the strict inequality in
Theorem~\ref{thm:main} can be improved, even for minimal systems.  We
begin with the identity action, where the obstruction reduces to
classical covering dimension.

\begin{proposition}\label{prop:sharp}
For every integer $D\geq1$, there is a topological dynamical system
$(K_D,\mathrm{id})$ with
\[
   \dim(K_D,\mathrm{id})=D
\]
that does not embed equivariantly into the $2D$-cubical shift.
\end{proposition}

\begin{proof}
Choose a $D$-dimensional compactum $K_D$ that does not topologically embed
into $\bbR^{2D}$.  For example, the $D$-skeleton of the $(2D+2)$-simplex
is a $D$-dimensional finite compact polyhedron and, by the
van Kampen--Flores theorem, does not embed into $\bbR^{2D}$
\cite[Theorem~5.1.1; see also Section~5.6]{Matousek2003}.  Give $K_D$
the identity action.

For the identity action, every probability measure is invariant.  Thus,
for every finite open cover $\cV$,
\[
   \ord(\cV,\mathrm{id})
   =\sup_{\mu\in\Prob(K_D)}\int_{K_D}\ord(\cV,x)\,d\mu(x)
   =\sup_{x\in K_D}\ord(\cV,x),
\]
where the last equality follows by taking Dirac measures.  Substitution
in the definition of dynamical dimension shows directly that
$\dim(K_D,\mathrm{id})$ is the usual Lebesgue covering dimension of
$K_D$.  The chosen polyhedron has covering dimension $D$
\cite{HurewiczWallman}; hence
\[
   \dim(K_D,\mathrm{id})=\dim K_D=D.
\]

If $F$ were an equivariant embedding into the $2D$-cubical shift, then
$\sigma F(x)=F(\mathrm{id}(x))=F(x)$ for every $x\in K_D$.  Hence
\[
   F(K_D)\subset\operatorname{Fix}(\sigma)
   =\{(\ldots,a,a,a,\ldots):a\in[0,1]^{2D}\}
   \cong[0,1]^{2D}.
\]
Thus $F$ would give a topological embedding of $K_D$ into
$[0,1]^{2D}\subset\bbR^{2D}$, contradicting the choice of $K_D$.
\end{proof}

Taking $n=2D$ in Proposition~\ref{prop:sharp} gives
$\dim(K_D,\mathrm{id})=n/2$ but no embedding into the $n$-cubical shift,
so the strict inequality in Theorem~\ref{thm:main} cannot be replaced by
a non-strict one.

The threshold is sharp even within the class of minimal systems.

\begin{proposition}\label{prop:minimal-sharp}
For every integer $n\geq1$, there exists a minimal topological
dynamical system $(X_n,T_n)$ such that
\[
   \dim(X_n,T_n)=\frac n2,
\]
but $(X_n,T_n)$ does not embed equivariantly into
$(([0,1]^n)^{\bbZ},\sigma)$.
\end{proposition}

\begin{proof}
By Lindenstrauss and Tsukamoto
\cite[Theorem~1.3]{LindenstraussTsukamoto2014}, there exists a minimal
topological dynamical system $(X_n,T_n)$ satisfying
\[
   \operatorname{mdim}(X_n,T_n)=\frac n2
\]
that does not embed equivariantly into $(([0,1]^n)^{\bbZ},\sigma)$.
Since the minimality implies the marker property, we have
$\dim(X_n,T_n)=\operatorname{mdim}(X_n,T_n)$
\cite[Theorem~9.2]{Meyerovitch2026}.
Therefore
\[
   \dim(X_n,T_n)=\operatorname{mdim}(X_n,T_n)=\frac n2,
\]
which proves the proposition.
\end{proof}

Thus the systems above are minimal, free, and have the marker
property; the strict one-half threshold is already optimal within this
class.

\section*{Acknowledgments}

The author is grateful to Tom Meyerovitch for many helpful discussions
and valuable comments that improved the manuscript.

\begingroup
\raggedright

\endgroup


\begin{thebibliography}{99}

\bibitem{Billingsley1999}
P.~Billingsley,
\emph{Convergence of Probability Measures}, 2nd ed.,
Wiley Series in Probability and Statistics, Wiley, New York, 1999.

\bibitem{DranishnikovLevin2025}
A.~Dranishnikov and M.~Levin,
\emph{A free $\bbZ$-action by isometries on a compact metric space which
is not embeddable into a cubical shift},
\href{https://arxiv.org/abs/2508.14628}{arXiv:2508.14628}, 2025.

\bibitem{Gromov1999}
M.~Gromov,
\emph{Topological invariants of dynamical systems and spaces of
holomorphic maps: I},
Math. Phys. Anal. Geom. \textbf{2} (1999), no.~4, 323--415,
\href{https://doi.org/10.1023/A:1009841100168}
{doi:10.1023/A:1009841100168}.

\bibitem{Gutman2015}
Y.~Gutman,
\emph{Mean dimension and Jaworski-type theorems},
Proc. London Math. Soc. (3) \textbf{111} (2015), no.~4, 831--850,
\href{https://doi.org/10.1112/plms/pdv043}
{doi:10.1112/plms/pdv043}.

\bibitem{GutmanLevinMeyerovitch2025}
Y.~Gutman, M.~Levin, and T.~Meyerovitch,
\emph{Equivariant embedding of finite-dimensional dynamical systems},
Math. Ann. \textbf{391} (2025), 915--936,
\href{https://doi.org/10.1007/s00208-024-02911-y}
{doi:10.1007/s00208-024-02911-y}.

\bibitem{GutmanQiaoTsukamoto2019}
Y.~Gutman, Y.~Qiao, and M.~Tsukamoto,
\emph{Application of signal analysis to the embedding problem of
$\bbZ^k$-actions},
Geom. Funct. Anal. \textbf{29} (2019), no.~5, 1440--1502,
\href{https://doi.org/10.1007/s00039-019-00499-z}
{doi:10.1007/s00039-019-00499-z}.

\bibitem{GutmanTsukamoto2014}
Y.~Gutman and M.~Tsukamoto,
\emph{Mean dimension and a sharp embedding theorem: extensions of
aperiodic subshifts},
Ergodic Theory Dynam. Systems \textbf{34} (2014), no.~6, 1888--1896,
\href{https://doi.org/10.1017/etds.2013.30}
{doi:10.1017/etds.2013.30}.

\bibitem{GutmanTsukamoto2020}
Y.~Gutman and M.~Tsukamoto,
\emph{Embedding minimal dynamical systems into Hilbert cubes},
Invent. Math. \textbf{221} (2020), no.~1, 113--166,
\href{https://doi.org/10.1007/s00222-019-00942-w}
{doi:10.1007/s00222-019-00942-w}.

\bibitem{Hall1935}
P.~Hall,
\emph{On representatives of subsets},
J. London Math. Soc. \textbf{10} (1935), no.~1, 26--30,
\href{https://doi.org/10.1112/jlms/s1-10.37.26}
{doi:10.1112/jlms/s1-10.37.26}.

\bibitem{HurewiczWallman}
W.~Hurewicz and H.~Wallman,
\emph{Dimension Theory},
Princeton Mathematical Series, vol.~4,
Princeton University Press, Princeton, NJ, 1941.

\bibitem{Lindenstrauss1999}
E.~Lindenstrauss,
\emph{Mean dimension, small entropy factors and an embedding theorem},
Publ. Math. Inst. Hautes \'Etudes Sci. \textbf{89} (1999), 227--262,
\href{https://doi.org/10.1007/BF02698858}
{doi:10.1007/BF02698858}.

\bibitem{LindenstraussTsukamoto2014}
E.~Lindenstrauss and M.~Tsukamoto,
\emph{Mean dimension and an embedding problem: an example},
Israel J. Math. \textbf{199} (2014), no.~2, 573--584,
\href{https://doi.org/10.1007/s11856-013-0040-9}
{doi:10.1007/s11856-013-0040-9}.

\bibitem{LindenstraussWeiss2000}
E.~Lindenstrauss and B.~Weiss,
\emph{Mean topological dimension},
Israel J. Math. \textbf{115} (2000), 1--24,
\href{https://doi.org/10.1007/BF02810577}
{doi:10.1007/BF02810577}.

\bibitem{Matousek2003}
J.~Matou\v{s}ek,
\emph{Using the Borsuk--Ulam Theorem: Lectures on Topological Methods in
Combinatorics and Geometry},
Universitext, Springer-Verlag, Berlin, 2003,
\href{https://doi.org/10.1007/978-3-540-76649-0}
{doi:10.1007/978-3-540-76649-0}.

\bibitem{Meyerovitch2026}
T.~Meyerovitch,
\emph{A new notion of dimension for dynamical systems and shift
embeddability},
Geom. Funct. Anal. \textbf{36} (2026), no.~3, 947--980,
\href{https://doi.org/10.1007/s00039-026-00742-4}
{doi:10.1007/s00039-026-00742-4}.

\bibitem{Shi2026}
R.~Shi,
\emph{Dynamical dimension and shift embeddability without the marker
property},
\href{https://arxiv.org/abs/2607.27880}{arXiv:2607.27880}, 2026.

\end{thebibliography}
\end{document}